\documentclass[a4paper,11pt]{article}
\usepackage[a4paper,vmargin={3.5cm,3.5cm},hmargin={2cm,2cm}]{geometry}
\usepackage[utf8]{inputenc}
\usepackage[french, english]{babel}
\usepackage[T1]{fontenc}   
\usepackage{mathpazo}
\usepackage{graphicx}
\usepackage{comment}
\usepackage{enumerate}
\usepackage{amsthm,amsmath,amsfonts,stmaryrd,mathrsfs,amssymb}
\usepackage{mathtools}
\usepackage{bbm, xcolor}
\usepackage{subfig}
\usepackage{hyperref}

\newcommand{\Ec}[1]{\mathbb{E} \left[#1\right]}

\newcommand{\Pp}[1]{\mathbb{P} \left(#1\right)}

\newcommand{\Ecsq}[2]{\mathbb{E} \left[#1\mathrel{}\middle|\mathrel{}#2\right]}

\newcommand{\Ppsq}[2]{\mathbb{P} \left(#1\mathrel{}\middle|\mathrel{}#2\right)}

\newcommand{\enstq}[2]{\left\lbrace#1\mathrel{}\middle|\mathrel{}#2\right\rbrace}
\newcommand{\indicator}[1]{\mathbbm{1}_{#1}}

\newcommand{\intervalle}[4]{\mathopen{#1}#2
	\mathclose{}\mathpunct{},#3
	\mathclose{#4}}

\newcommand{\intervalleof}[2]{\intervalle{(}{#1}{#2}{]}}
\newcommand{\intervallefo}[2]{\intervalle{[}{#1}{#2}{)}}
\newcommand{\intervalleoo}[2]{\intervalle{(}{#1}{#2}{)}}

\newcommand{\MAST}{\mathrm{MAST}}
\newcommand{\Tt}{\mathcal{T}}
\newcommand{\Rr}{\mathcal{R}}
\newcommand{\Bb}{\mathcal{B}}

\newcommand{\ii}{\mathbf{i}}
\newcommand{\jj}{\mathbf{j}}
\newcommand{\ikx}{\mathbf{i}_k(x)}
\newcommand{\ijx}{\mathbf{i}_j(x)}
\newcommand{\Dir}{\mathrm{Dir}}
\newcommand{\diam}{\mathrm{diam}}

\newcommand{\eps}{\varepsilon}

\newtheorem{theorem}{Theorem}
\newtheorem{definition}[theorem]{Definition}
\newtheorem{proposition}[theorem]{Proposition}
\newtheorem{corollary}[theorem]{Corollary}
\newtheorem{lemma}[theorem]{Lemma}
\newtheorem{remark}[theorem]{Remark}
\newtheorem*{claim*}{Claim}

\hypersetup{
	colorlinks   = true, %Colours links instead of ugly boxes
	urlcolor     = blue, %Colour for external hyperlinks
	linkcolor    = blue, %Colour of internal links
	citecolor   = red %Colour of citations
}

\begin{document}
\title{Maximum Agreement Subtrees and Hölder homeomorphisms between Brownian trees}
\author{
	Thomas \textsc{Budzinski}\thanks{ ENS de Lyon and CNRS.\hfill  \href{mailto:thomas.budzinski@ens-lyon.fr}{\texttt{thomas.budzinski@ens-lyon.fr}}}
	\qquad\&\qquad
	Delphin \textsc{S\'enizergues}\thanks{MODAL'X, UPL, Univ. Paris Nanterre, CNRS, F92000 Nanterre France.\hfill  \href{mailto:dsenizer@parisnanterre.fr}{\texttt{dsenizer@parisnanterre.fr}}}
}

\maketitle

\abstract{We prove that the size of the largest common subtree between two uniform, independent, leaf-labelled random binary trees of size $n$ is typically less than $n^{1/2-\eps}$ for some $\eps>0$. Our proof relies on the coupling between discrete random trees and the Brownian tree and on a recursive decomposition of the Brownian tree due to Aldous. Along the way, we also show that almost surely, there is no $(1-\eps)$-H\"older homeomorphism between two independent copies of the Brownian tree.}

\section{Introduction}

\paragraph{Maximum agreement subtree.}
Let $t, t'$ be two binary trees with $n$ leaves labelled from $1$ to $n$. The \emph{maximum agreement subtree} of $t$ and $t'$ is the size of the largest subset $I \subset \{1,\dots,n\}$ such that the subtrees of $t$ and $t'$ induced by the labels of $I$ are the same (as on Figure~\ref{fig:example mast}, see also Section~\ref{subsec_discrete_defns} for precise definitions). This quantity, which will be denoted by $\MAST(t,t')$, was introduced by Gordon and Finden~\cite{Gor80,FG85} in order to measure the compatibility of the outputs of different classifications methods in phylogeny. 
It is also a generalization of the well studied problem of the longest increasing subsequence of a permutation, and the two problems share a lot of similarities (as noted e.g. in~\cite{aldous_largest_2022}).
\begin{figure}
	\begin{tabular*}{15cm}{ccc}
		\subfloat[A tree $t$ and its induced subtree $t|_I$\label{subfig:a}]{\includegraphics[page=1]{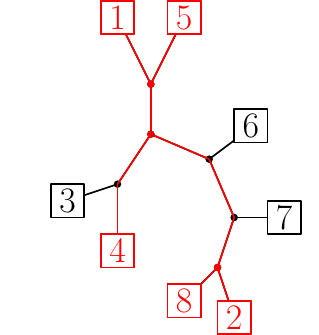}}&
		\subfloat[The tree $t|_I=t'|_I$ \label{subfig:b}]{\includegraphics[page=2]{MAST_example}} &
		\subfloat[A tree $t'$ and its induced subtree $t'|_I$\label{subfig:a}\label{subfig:c}]{\includegraphics[page=3]{MAST_example}}
	\end{tabular*}
	\caption{Two labelled binary trees $t$ an $t'$  and their largest common subtree, induced by the set $I=\{1,2,4,5,8\}$.}\label{fig:example mast}
\end{figure}
Since then, it has been studied from algorithmic, extremal and probabilistic points of view. 
In particular, the quantity $\MAST(t,t')$ can be computed in polynomial time in $n$~\cite{SW93}. 
On the extremal side, the minimal possible values of $\MAST(t,t')$ over all pairs $(t,t')$ of leaf-labelled binary trees of size $n$ is known to be of order $\log n$ (the upper bound was proved in~\cite{KKM92} and the lower bound in~\cite{Mark20}).

\paragraph{Maximum agreement subtree of random trees.}
Another natural question is to understand the \emph{typical} order of magnitude of the maximum agreement subtree, that is, the random variable $\MAST(T_n, T'_n)$, where $T_n$ and $T'_n$ are random trees of size $n$. 
The most natural model is the one where $T_n$ and $T'_n$ are independent and picked uniformly in the set of labelled binary trees of size $n$. 
This model was first investigated by Bryant, McKenzie and Steel~\cite{bryant_size_2003}, who proved by a first moment computation that $\MAST(T_n, T'_n)$ is $O(\sqrt{n})$ with high probability. 
They also provided numerical evidence that $\MAST(T_n, T'_n)$ should be of order $n^{\beta}$ for some $\beta$ close to $\frac{1}{2}$. 
On the other hand, a polynomial lower bound of order $n^{1/8}$ was obtained by Bernstein, Ho, Long, Steel, St. John and Sullivant in~\cite{bernstein_bounds_2015}. This lower bound was recently improved to $n^{\frac{\sqrt{3}-1}{2}}\approx n^{0.366}$ by Aldous~\cite{aldous_largest_2022} and to $n^{0.4464}$ by Khezeli~\cite{Khez22} in expectation. Finally, we also mention that $\sqrt{n}$ has been proved to be the right order of magnitude if the trees $T_n$ and $T'_n$ are conditioned to have the same shape~\cite{misra_bounds_2019}, and that the upper bound in $\sqrt{n}$ holds robustly for many random trees models arising from branching processes~\cite{Pit23}.

Our main contribution in this paper is to show that the upper bound $\sqrt{n}$ is actually not optimal in the independent model, which was conjectured by Aldous in~\cite{aldous_largest_2022}.

\begin{theorem}\label{thm_MAST_upper}
	For all $n \geq 3$, let $T_n$ and $T'_n$ be two independent uniform labelled binary trees of size $n$. There exists a constant $\eps_{\ref{thm_MAST_upper}}>0$ such that we have
	\[ \Pp{\MAST (T_n,T'_n) < n^{1/2-\eps_{\ref{thm_MAST_upper}}}} \xrightarrow[n \to +\infty]{} 1.\]
\end{theorem}

More explicitly, we find that we can take $\eps_{\ref{thm_MAST_upper}}=10^{-338}$ (see Section~\ref{subsec:explicit constants} for a discussion on explicit constants). We have not tried to optimize the constants and this value should be easy to improve, but we do not think that our strategy of proof can give a "reasonable" lower bound (like e.g. $\eps_{\ref{thm_MAST_upper}}=10^{-6}$). We also mention that our arguments are sufficient to prove that the probability that $\MAST(T_n,T'_n)$ exceeds $n^{1/2-\eps_{\ref{thm_MAST_upper}}}$ is $O(n^{-a})$ for some $a>0$, and that $\Ec{\MAST(T_n,T'_n)} \leq n^{1/2-\eps}$ for some $\eps>0$ (see Section~\ref{subsec:discussion} for a quick discussion).

\paragraph{Comparison with the Brownian tree.}
As recalled before, it is proved in~\cite{misra_bounds_2019} that the MAST of two trees of the same shape is typically of order $\sqrt{n}$. 
Therefore, our strategy relies on the fact that two independent large random trees have "different shapes" at every scale.
To formalize this, we make heavy use of the continuous scaling limit of $T_n$, which exhibits nice scale invariance properties, and on which more explicit computations can be performed.

More precisely, we denote by $\Tt$ the \emph{Brownian tree}, which is the scaling limit of $T_n$, seen as a measured metric space, where distances have been normalized by $\sqrt{n}$ and masses by $n$. 
This compact, continuous random tree with fractal dimension 2 was introduced in~\cite{Ald91} and can be built in a natural way from a normalized Brownian excursion (see Section~\ref{subsec:BCRT} for complete definitions). It also has the important property that its branching points all have degree $3$. We highlight that comparisons between the discrete trees $T_n$ and the continuous object $\Tt$ already play an important role in the proofs of the lower bounds of~\cite{aldous_largest_2022} and~\cite{Khez22}.

\paragraph{H\"older homeomorphisms of Brownian tree.}
Since proving Theorem~\ref{thm_MAST_upper} requires to compare the shapes of two independent copies of $\Tt$, we obtain along the way the following result of independent interest.

\begin{theorem}\label{thm_holder_homeo}
	Let $\Tt$ and $\Tt'$ be two independent copies of the Brownian tree. There exists a constant $\eps_{\ref{thm_holder_homeo}}>0$ such that almost surely, there is no $(1-\eps_{\ref{thm_holder_homeo}})$-H\"older homeomorphism from $\Tt$ to $\Tt'$.
\end{theorem}

Just like in Theorem~\ref{thm_MAST_upper}, we find that we can take $\eps_{\ref{thm_holder_homeo}}=10^{-338}$, which we did not try to optimize.
Although none of Theorems~\ref{thm_MAST_upper} and~\ref{thm_holder_homeo} easily implies the other, they are closely related to each other. 
Indeed, as can be seen on Figure~\ref{fig:example mast}, a common subtree of two trees $T_n$ and $T'_n$ gives a "correspondence" between a part of $T_n$ and a part of $T'_n$, which can be extended to a homeomorphism in the continuous limit.
This is not a completely new idea, as the arguments of~\cite{aldous_largest_2022} (and the improvements done in~\cite{Khez22}) can already be interpreted as a proof of the existence of a homeomorphism from $\Tt$ to $\Tt'$ with a certain H\"older exponent. 
As we check in Theorem~\ref{thm:existence_holder_homeomorphism} in the Appendix below, the actual H\"older exponent given by~\cite{aldous_largest_2022} turns out to be $5-2\sqrt{6}\approx 0.1010$.
More generally, statements similar to Theorem~\ref{thm_holder_homeo} on very different objects appear under the name of \emph{H\"older equivalence} in the geometry literature. 
In geometry, this problem is often of the following form: given a metric space $X$ that is homeomorphic to $\mathbb{R}^n$, what is the optimal H\"older exponent of a homeomorphism from $\mathbb{R}^n$ to $X$? 
An immediate upper bound is $\frac{n}{\dim_H(X)}$, where $\dim_H(X)$ is the Hausdorff dimension of $X$. 
We refer to~\cite{Gromov96} for improved upper bounds in specific contexts such as sub-Riemannian or contact manifolds. 
%Contrary to the arguments found in the geometry literature, which uses elaborate 
However, we are in a very different setting here, as the Brownian trees involved are not manifolds, and so our arguments do not share any commonalities. 
Another difference between our setting and the one studied by Gromov is that we prove that the H\"older exponent cannot be arbitrarily close to $1$ in a context where both sides of the homeomorphism have the same Hausdorff dimension.

We also note that Theorem~\ref{thm_holder_homeo} becomes quite easy if "$(1-\eps)$-H\"older" is replaced by "Lipschitz"\footnote{For example, take a large branching point $b_1$ of $\Tt$, and consider an "exceptional scale" $\delta$ at which $b_1$ is unusually close to another branching point $b_2$.
Then $\Psi(b_1)$ would have to be a large branching point of $\Tt'$, and $\Psi(b_2)$ would have to be a branching point of $\Tt'$ (of scale $\approx \delta$) very close to $\Psi(b_1)$, which is unlikely to exist.}
Finally, we remark that our results have a similar flavor to those proved in~\cite[Theorem~1.2, Theorem~1.7]{BBDFGMP22} for a quite different model (largest increasing subsequence of a random separable permutation). More precisely, the proofs in~\cite{BBDFGMP22} consist of showing that a random tree cannot contain a large subtree satisfying some properties, which improves on the first moment upper-bound and is achieved by comparison with continuous objects.
The very recent preprint \cite{borga_power_2023} improves their result (and also provides some lower-bound) using some careful analysis on the Brownian tree and its associated fragmentation process.

\paragraph{Recursive decomposition of the Brownian tree.}
In order to highlight more precisely what Theorem~\ref{thm_MAST_upper} and Theorem~\ref{thm_holder_homeo} have in common, we introduce an important tool in our proofs, which already crucially appears in~\cite{aldous_largest_2022}. 
This is a randomized recursive decomposition of the Brownian tree $\Tt$, which was introduced by Aldous~\cite{aldous_recursive_1994}. 
The decomposition consists in picking $3$ random uniform points $X_1, X_2, X_3$ in $\Tt$, blowing up $\Tt$ into three pieces at the unique branching point that separates $X_1, X_2, X_3$, and iterating the decomposition in each of the three pieces (see Section~\ref{subsec:aldous recursive decomposition} for complete definitions). 
After $k$ steps, we obtain a (randomized) partition of $\Tt$ into $3^k$ regions, indexed by a set $\mathbb{T}_3^k$. 
We denote those regions by $\left( \Rr[\ii] \right)_{\ii \in \mathbb{T}_3^k}$. 
This decomposition enjoys very nice independence properties that we will heavily rely on.

We can now state the key result that we will use to prove both Theorem~\ref{thm_MAST_upper} and Theorem~\ref{thm_holder_homeo}. 
It roughly states that a homeomorphism between two independent realizations $\Tt, \Tt'$ of the Brownian tree cannot be "almost measure-preserving", in the sense that it has to send "most" regions of $\Tt$ to regions of $\Tt'$ with a much smaller mass. We will denote by $|A|$ the measure of a subset $A$ of a Brownian tree.
\begin{proposition}\label{prop_homeo_mass}
	Let $\Tt$ and $\left( \Rr[\ii] \right)_{\ii \in \mathbb{T}_3^k}$ be a Brownian tree and its recursive decomposition, and let $\Tt'$ be an independent copy of $\Tt$. There exist constants $\xi, \eta>0$ such that almost surely the following holds for $k$ large enough:
	For any homeomorphism $\Psi : \Tt \to \Tt'$, 
	if we define the set $U^{k,\Psi} \subset \mathbb{T}_3^k$ of indices as
	\begin{align*}
		U^{k,\Psi}:= \enstq{\ii \in \mathbb{T}_3^k}{\left| \Psi \left( \Rr[\ii] \right) \right| > e^{-\eta k} \left| \Rr[\ii] \right|},
	\end{align*}
	then the subset 
	$
	E^{k,\Psi}:= \bigcup_{\ii \in U^{k,\Psi}} \Rr[\ii]$ of $\Tt
	$
	has measure at most $e^{-\xi k}$.
\end{proposition}
Theorem~\ref{thm_holder_homeo} will follow almost immediately from this result. 
On the other hand, in order to deduce Theorem~\ref{thm_MAST_upper} from Proposition~\ref{prop_homeo_mass}, we will rely on the nice coupling existing between the discrete random tree $T_n$ and the continuous one $\Tt$. 
We can then argue that a common subtree between $T_n$ and $T'_n$ can be extended to a homeomorphism $\Psi$ from $\Tt$ to $\Tt'$. 
Proposition~\ref{prop_homeo_mass} guarantees that for most of the regions $\Rr[\ii]$, either $\Rr[\ii]$ or $\Psi(\Rr[\ii])$ is very small, so only few labels can appear in both $\Rr[\ii]$ and in $\Psi(\Rr[\ii])$ when $T_n$ is coupled to $\Tt$ and $T'_n$ to $\Tt'$. We will conclude by using the classic square root upper bound for each of those regions (Lemma~\ref{lem:bound mast of two regions}).

\paragraph{Ideas of the proof of Proposition~\ref{prop_homeo_mass}.}
Finally, let us mention some of the ideas behind the proof of Proposition~\ref{prop_homeo_mass}. The proof roughly consists in showing that a certain multiscale exploration of the tree $\Tt$ has many "mismatches" with the analog exploration in $\Tt'$, which we believe is the main innovation of the present work.
Fix a typical point $x \in \Tt$, and imagine that we try to build a "good" homeomorphism from $\Tt$ to $\Tt'$. By looking at smaller and smaller regions of the recursive decomposition around the point $x$, we can encode the masses of a sequence of nested neighbourhoods of $x$ by a sequence $\left(f_j(x) \right)_{j \geq 0}$ of 
i.i.d.\ numbers, where $j$ represents decreasing scales.
We will argue that it is not possible to find $x' \in \Tt'$ such that $f_j(x')$ is very close to $f_j(x)$ for most of the scales $j$. 
This will imply that the ratio between the mass of a small region around $x$ and the mass of its image around $\Psi(x)$ cannot stay "stationary" as the scale of that region decreases to $0$.
By a "martingale-like" argument, we will conclude that this ratio must decay quickly, which will yield Proposition~\ref{prop_homeo_mass}. This is somewhat reminiscent of some ideas of~\cite{BBDFGMP22}, in the sense that "finding a large substructure is difficult because some positive proportion of the mass is lost at every scale".

\paragraph{Structure of the paper.}
In Section~\ref{sec:definitions and preliminaries}, we will introduce all the definitions of discrete and continuous objects that will be needed in the proofs, as well as some basic properties of the Brownian tree and of its recursive decomposition. 
Section~\ref{sec:proof of prop homeo mass} is the central part of the paper and is devoted to the proof of Proposition~\ref{prop_homeo_mass}, which represents most of the work. 
In Section~\ref{sec:proof of main results}, we conclude the proofs of the main theorems. 
In Section~\ref{sec:discussion of the results}, we discuss the quantitative values of $\eps_{\ref{thm_MAST_upper}}$ and $\eps_{\ref{thm_holder_homeo}}$ provided by our proof, as well as some open questions.
In Appendix~\ref{app:construction of homeomorphism} , we construct a Hölder homeomorphism between $\Tt$ and $\Tt'$.

\paragraph{Acknowledgments.} The authors are grateful to Omer Angel 
for many interesting discussions about various aspects of the model studied in this paper. They would also like to thank Nicolas Curien for useful discussions, and Valentin F\'eray for noticing the analogy with~\cite{BBDFGMP22, borga_power_2023}.
They also thank the anonymous referees for their constructive comments and remarks.
D.S.'s research has been conducted within the FP2M federation (CNRS FR 2036).

\tableofcontents

\section{Definitions and preliminaries}\label{sec:definitions and preliminaries}

\subsection{Labelled binary trees and maximum agreement subtree}\label{subsec_discrete_defns}

A \emph{finite tree} is a finite, connected graph with no cycles. 
A \emph{finite binary tree} $t$ is a finite tree where all the vertices have degree either $3$ or $1$. 
The vertices of degree $3$ will be called the \emph{nodes} of $t$, and the vertices of degree $1$ will be called its \emph{leaves}. 
A \emph{labelled binary tree of size $n \geq 2$} is a finite binary tree with exactly $n$ leaves (and therefore $n-2$ nodes), some of which are labelled by integers so that each label appears at most once. By convention, we also say that a single vertex with no edge is a binary tree of size $1$, and the empty tree is a binary tree of size $0$.

For $n \geq 1$, we denote by $\Bb_n$ the set of such trees where all $n$ leaves are labelled and the labels are $1, \dots, n$, up to isomorphism (these are sometimes called \emph{cladograms} in the literature).
We highlight that the trees that we consider are not rooted, and they are not plane trees (i.e. there is no clockwise ordering of the neighbours of a fixed vertex).

We recall that for all $n \geq 3$, we have
\[ \# \Bb_n = (2n-5)!!=1 \cdot 3 \cdot \dots \cdot (2n-7) \cdot (2n-5). \]
In all the paper, we will denote by $T_n$ a uniform random variable on $\Bb_n$. 
We will also denote by $T'_n$ an independent copy of $T_n$. If $t \in \Bb_n$ and $I$ is a subset of $\{1,\dots,n\}$, we will denote by $t|_I$ the \emph{subtree of $t$ induced by $I$}. 
More precisely, this is the labelled binary tree obtained from $t$ by keeping only the paths joining the labels of $I$ together, and by contracting the vertices of degree $2$ that may appear in the process (see Figure~\ref{fig:example mast}). Note that $t|_I$ does not belong to $\Bb_{\#I}$, unless $I=\{1,\dots,k\}$. If $t,t' \in \Bb_n$, we write
\[ \MAST(t,t')=\max \left\{ \#I \, \big|\, I \subset \{1,\dots,n\} \mbox{ such that } t|_I=t'|_I \right\},\]
where by $t|_I=t'|_I$ we mean that $t|_I$ and $t'|_I$ are isomorphic as leaf-labelled trees.

If $v_1, v_2$ are nodes of a finite binary tree, a \emph{region $r$ of $t$ delimited by $v_1, v_2$} is a connected component of the forest obtained by blowing up the nodes $v_1$ and $v_2$ into three leaves each. In particular, it is a labelled binary tree, where the leaves coming from $v_1$ or $v_2$ are unlabelled. 
We write $\# r$ for the number of original leaves of $t$ that belong to the region $r$ and call this quantity the \emph{size} of $r$.
%By convention, we also say that a region of size $1$ consists of a single labelled leaf of $t$ with no edge, and a region of size $0$ is the empty subtree.
If $t,t'$ are two labelled binary trees and if $r \subset t$ and $r' \subset t'$ are two such regions, we may consider the quantity $\MAST(r,r')$, which is the size of the largest subset $I$ of $\{1,\dots,n\}$ such that all the elements of $I$ appear both in $r$ and in $r'$, and $r|_I=r'|_I$.

\subsection{The Brownian tree}\label{subsec:BCRT}

We start with the construction of the Brownian tree $\Tt$ introduced in~\cite{Ald91, Ald93}. Let $\mathbf{e}=(e_t)_{0 \leq t \leq 1}$ be a normalized Brownian excursion (that is, a Brownian motion conditioned to stay positive in $(0,1)$ and hit $0$ at time $1$). For $s,t \in [0,1]$, we write $s \sim_{\mathbf{e}} t$ if
\[ e_s=e_t=\min_{[s \wedge t, s \vee t]} e_u,\]
where $s \wedge t$ and $s \vee t$ stand respectively for $\min(s,t)$ and $\max(s,t)$. 
We also write
\[ d_{\mathbf{e}}(s,t)=e_s+e_t-2 \min_{[s \wedge t, s \vee t]} e_u. \]
Then $d_{\mathbf{e}}$ is a pseudo-distance on $[0,1]$, i.e. it is symmetric and satisfies the triangle inequality. Moreover, we have $d_{\mathbf{e}}(s,t)=0$ if and only if $s \sim_{\mathbf{e}} t$. 
Then the quotient $\Tt=[0,1] / \sim_{\mathbf{e}}$ equipped with $d_{\mathbf{e}}$ is a random compact metric space, which we call the \emph{Brownian tree}. 
Moreover $\Tt$ carries a natural probability measure, which is the pushforward of the Lebesgue measure on $[0,1]$ under the canonical map from $[0,1]$ to $\Tt$. 
We will denote by $|\cdot |$ this probability measure on $\Tt$, which has full support and no atom. 
We recall that the metric space $\Tt$ is almost surely a real tree, i.e. for all $x,y$ in $\Tt$, there is a unique injective, continuous path from $x$ to $y$, and this path is a geodesic. 
For $m>0$, we say that a random measured metric space $(\mathcal{E},d,\mu)$ is a \emph{Brownian tree of mass $m$} if $(\mathcal{E}, m^{-1/2}d, m^{-1}\mu)$ has the law of $\Tt$. 
Note in particular that $\Tt$ is a Brownian tree with mass $1$.

Let $k \geq 1$ and let $(\mathcal{E},d,\mu)$ be a Brownian tree with mass $m$. Conditionally on $(\mathcal{E},d,\mu)$, let $X_1, \dots, X_k$ be $k$ i.i.d.\ points of $\mathcal{E}$, sampled according to (a normalized version of) $\mu$. Note that almost surely, the points $X_i$ are all leaves, i.e. $\mathcal{E} \backslash \{X_i\}$ is connected. We call $(\mathcal{E},d,\mu,X_1,\dots X_k)$ a \emph{$k$-pointed Brownian tree with mass $m$}.

In this work, decompositions of $\Tt$ into several regions will play a crucial role. For this purpose, we will call \emph{region of $\Tt$} the closure of a connected component of $\Tt \backslash F$, where $F$ is a subset of $\Tt$ of cardinal $0$, $1$ or $2$.

We recall that $\Tt$ is almost surely a \emph{binary} real tree, i.e. for all $x \in \Tt$, the space $\Tt \backslash \{x\}$ has at most three connected components. A point $b \in \Tt$ such that $\Tt \backslash \{ b \}$ has exactly three connected components will be called a \emph{branching point} of $\Tt$. Moreover, we define the \emph{size} of a branching point $b$ as the measure of the smallest of the three connected components of $\Tt \backslash \{b\}$, and denote this quantity by $|b|_{\Tt}$. 
Similarly, if $\Rr$ is a region of $\Tt$ and $b \in \Rr$ is such that $\Rr \backslash \{b\}$ has three connected components, the \emph{relative size} of $b$ in $\Rr$ is the measure of the smallest such connected component, and is denoted by $|b|_{\Rr}$.
\subsection{Dirichlet and Beta distributions}
We recall here the definition of the Dirichlet and Beta distributions. 
For $(a_1,\dots, a_d)\in \intervalleoo{0}{\infty}^d$ the \emph{Dirichlet distribution} $\Dir \left(a_1, \dots, a_d\right)$ is the probability measure on the simplex
\[ \left\{ (m_1, \dots, m_{d}) \ \bigg| \ \forall 1\leq i \leq d, \quad  m_i \geq 0 \quad  \text{and} \quad  \sum_{i=1}^{d} m_i=1 \right\}\]
with density
\[ \frac{\Gamma(a_1+\dots + a_d)}{\Gamma(a_1)\dots \Gamma(a_d)}  \prod_{i=1}^{d} m_i^{a_i-1} \]
with respect to the Lebesgue measure. 
For $(W_1,W_2)\sim \Dir\left(a,b\right)$, the first coordinate $W_1$ is said to have the $\mathrm{Beta}(a,b)$ distribution.

The Dirichlet distribution enjoys two properties that we use throughout the paper.
The first one can be derived from the so-called \emph{beta-gamma algebra} results developed in \cite{dufresne_algebraic_1998}; for the second one we refer to \cite[Lemma~17]{addario-berry_critical_2010}.
Suppose $(W_1,\dots,W_d) \sim \Dir \left(a_1, \dots, a_d\right)$. 
Then 
\begin{itemize}
	\item For any $1\leq i\leq d-1$, we have 
	\begin{align}\label{eq:dirichlet distrib factorisation}\left\lbrace
		\begin{aligned}
		W_1+\dots + W_i &\sim \mathrm{Beta}(a_1+\dots +a_i,a_{i+1}+\dots + a_d) \\ 
	    (W_1+\dots + W_i)^{-1}\cdot (W_1,\dots ,W_i) &\sim \Dir \left(a_1, \dots, a_i\right)\\
	    (W_{i+1}+\dots + W_d)^{-1}\cdot (W_{i+1},\dots ,W_d) &\sim \Dir \left(a_{i+1}, \dots, a_d\right)
    \end{aligned}
\right.
	\end{align}
and the three random variables appearing in the last display are independent.
	\item Let $I$ be such that $\Ppsq{I=i}{(W_1,\dots,W_d)}=W_i$. Then for any $i\in \{1,\dots,d\}$ we have
	\begin{align}\label{eq:dirichlet distrib size-biaising} 
		\Pp{I=i}=\frac{a_i}{a_1+\dots+a_d} \quad \text{and} \quad \mathrm{Law}\left((W_1,\dots,W_d) \ \vert \ I=i \right)= \Dir \left(a_1, \dots,a_i+1,\dots, a_d\right).
	\end{align} 
\end{itemize}
\subsection{Coupling between discrete and continuum random trees}\label{subsec:discrete_continuous_coupling}
\begin{figure}
	\begin{tabular*}{15cm}{cc}
		\subfloat[A realization of $T_5$ \label{subfig:a}]{\includegraphics[page=1,scale=1]{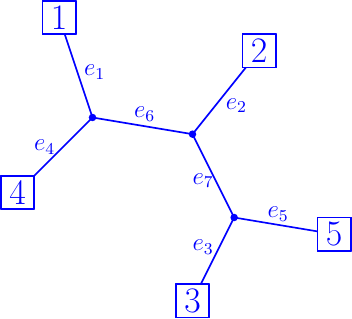}}&
		\subfloat[Patching together rescaled copies of independent Brownian trees along the structure of $T_5$ \label{subfig:b}]{\includegraphics[page=2,scale=0.8]{decomposition_Tt_along_Tn}} \\
		&
		\subfloat[The obtained $5$-pointed tree\label{subfig:c}]{\includegraphics[page=3, scale=0.8]{decomposition_Tt_along_Tn}}
	\end{tabular*}
	\caption{Coupling between the $5$-pointed Brownian tree $(\widetilde{\Tt},X_1^1,\dots, X_5^1)$ and $T_5$. The combinatorial structure of the paths joining the distinguished points, shown in blue, is that of the discrete tree $T_5$ we started with.}\label{fig:decomposition}
\end{figure}
We recall the classical coupling between the discrete tree $T_n$ and the continuous tree $\Tt$. 
Let $n \geq 3$ and let $(e_i)_{1 \leq i \leq 2n-3}$ be an enumeration of the edges of $T_n$ such that for $1 \leq i \leq n$, the edge $e_i$ is the unique edge incident to the leaf labelled $i$. 

Let $(W_i)_{1 \leq i \leq 2n-3}$ be a random vector with distribution $\Dir \left( \frac{1}{2}, \dots, \frac{1}{2}\right)$, independent from $T_n$. 
Conditionally on $T_n$ and $(W_i)$, let $\left( \Tt_i, X_i^1, X_i^2 \right)_{1 \leq i \leq 2n-3}$ be independent bi-pointed Brownian trees with respective masses $W_i$. 
For all $i$, let $v_i^1$ and $v_i^2$ be the two endpoints of the edge $e_i$ of $T_n$, with the convention that if one of the endpoints is a leaf, then it is $v_i^1$. For $1 \leq i,i' \leq 2n-3$ and $j,j' \in \{1,2\}$, we write $X_i^j \sim X_{i'}^{j'}$ if $v_i^j=v_{i'}^{j'}$, and write
\[ \widetilde{\Tt} = \left( \bigcup_{i=1}^{2n-3} \Tt_i \right) / \sim.\]
As a metric space, this quotient is understood as the "metric gluing" of the $\Tt_i$'s in the sense of \cite{burago_course_2001}. 
The measure on $\Tt$ is straightforwardly obtained from those of the $\widetilde{\Tt}_i$, and its total mass is $\sum_{i=1}^{2n-3} W_i=1$.

The following result can be found for $n=3$ in \cite{aldous_recursive_1994}. 
Even though the corresponding result for $n>3$ is folklore, we were not able to find a statement in the literature that exactly matches the one that we use here. However, it can be seen as a consequence of the discrete result proven in \cite[Exercise~7.4.12 and Exercise~7.4.13]{pitman_combinatorial_2006} (the proof relies on the R\'emy algorithm~\cite{Rem85} to build uniform binary trees, and the Dirichlet distribution comes from a P\'olya urn argument). 
\begin{theorem}\label{thm_discrete_continuous_coupling}\cite{aldous_recursive_1994,pitman_combinatorial_2006}
	The $n$-pointed measured metric space $\left( \widetilde{\Tt}, X_1^1, \dots, X_n^1 \right)$ has the law of an $n$-pointed Brownian tree with mass $1$.
\end{theorem}
See Figure~\ref{fig:decomposition} for an illustration of this construction.
In particular, the combinatorial structure of the paths joining $n$ uniform points $X_1, \dots, X_n$ of $\Tt$ is that of $T_n$.
In the rest of the paper, we will always consider that the continuous tree $\Tt$ and the discrete tree $T_n$ are coupled in this way.

We conclude this section with a lemma comparing the mass of a region in $\Tt$ and some corresponding region in $T_n$.
\begin{lemma}\label{lem:comparison mass vs leaf-count}
	Let $\eps>0$ and let $\left( \Tt, X_1, \dots, X_n \right)$ be an $n$-pointed Brownian tree. 
	Then, with probability $1 - o_n(1)$, for any region $\Rr$ of $\Tt$ delimited by at most two branching points, 
	denoting $R$ the smallest region of $T_n$ that contains all the leaves with label $j \in \{1,\dots,n\}$ such that $X_j \in \Rr$, we have the following bound 
\begin{align*}
	\#R \leq n^\eps \vee (n^{1+\eps} \cdot |\Rr|).
\end{align*}
\end{lemma}

\begin{proof}
	Let $\Rr$ be a region of $\Tt$ delimited by at most two branching points. 
	This way the topological frontier $\partial \Rr$ of the region $\Rr$ contains at most two elements. 
	With the notation just above, for any $1 \leq i \leq 2n-3$, we denote by $\overset{\circ}{\Tt_i}$ the interior of the region $\Tt_i$, seen as a subset of $\Tt$. 
	We introduce 
	\begin{align*}
		I=\enstq{i\in \{1,\dots,2n-3\}}{\overset{\circ}{\Tt_i} \cap \Rr \neq \emptyset }.
	\end{align*} 
We note that the set of edges $\{e_i, i\in I\}$ defines a region $\tilde{R}$ of $T_n$ delimited by at most two nodes; we denote $N:=\#\tilde{R}$. 
By definition of $I$, the region $\tilde{R}$ contains all the leaves with labels $j \in \{1,\dots,n\}$ that are such that $X_j \in \Rr$, so $R\subset \tilde{R}$ and so $\#R\leq \#\tilde{R}=N$. 
If $N<n^\eps$ we have nothing to prove, so let us assume $N \geq n^{\eps}$. 
Remark that since $\tilde{R}$ is a region of $T_n$, it has at most two unlabelled leaves, so its number of edges satisfies $\#I\in \{2N-3,2N-1,2N+1\}$, depending on whether it is delimited by $0$ or $1$ or $2$ nodes. 
 
Now, for any $i\in I$, if $\overset{\circ}{\Tt_i} \cap \Rr \neq \overset{\circ}{\Tt_i}$ then it means that $(\partial \Rr)\cap \Tt_i \neq \emptyset$. 
Since $\partial \Rr$ has cardinality at most $2$, this can happen for at most two such $i\in I$, say $i_1$ and $i_2$ (pick them arbitrarily if only $0$ or $1$ such value exists).
This entails $\Rr\supset \bigcup_{i\in I \setminus \{i_1,i_2\}} \Tt_i$.

Now, let $(W_i)_{1\leq i \leq 2n-3}$ be as in the coupling of Theorem~\ref{thm_discrete_continuous_coupling} (that is, $W_i$ is the mass of the set of those points $x \in \Tt$ such that $e_i$ is the closest edge of $T_n$ from $x$). 
From the above reasoning, we have
\begin{equation}\label{eqn:mass_region_and_W}
	\left| \Rr \right|\geq \sum_{i \in I\setminus\{i_1,i_2\}} W_i.
\end{equation}

We now condition on the discrete tree $T_n$. We recall that $(W_i)_{1\leq i \leq 2n-3}$ is independent of $T_n$ and has Dirichlet distribution, so according to \eqref{eq:dirichlet distrib factorisation} we have 
	\begin{align*}
	\sum_{i\in I \setminus \{i_1,i_2\}} W_i 
		\sim \mathrm{Beta}\left(\frac{\#I -2}{2},\frac{2n-1-\#I}{2}\right)
	\end{align*}
	conditionally on $T_n$. Let us now fix the region $\tilde{R}$ of $T_n$. Using the explicit density of the $\mathrm{Beta}$ distribution, we can write, recalling that $N = \# \tilde{R} \geq n^\eps$ and $2N-3 \leq \# I \leq 2N+1$:
	\begin{align*}
		\Ppsq{\sum_{i\in I \setminus \{i_1,i_2\}} W_j \leq  n^{-1-\eps}\cdot  N}{T_n}&= \frac{\Gamma\left(\frac{2n-3}{2}\right)}{\Gamma\left(\frac{\#I -2}{2}\right)\Gamma\left(\frac{2n-1-\#I}{2}\right)} \int_0^{n^{-1-\eps}\cdot N} x^{\frac{\#I -4}{2}} \cdot (1-x)^{\frac{2n-3-\#I}{2}} \mathrm{d}x.\\
		&\leq \frac{n^{\#I/2+1}}{(\#I/2)^{(\#I)/2} e^{O(\#I)}} \cdot (n^{-1-\eps} N)^{\frac{\# I-4}{2}}\\
		&\leq n^{-\eps n^{\eps}+O(\frac{n^\eps}{\log n})},
	\end{align*}
	which decays faster than any polynomial in $n$. 
	In the above computation we have used the fact that $\frac{\Gamma(x)}{\Gamma(x-k)}\leq x^k$ for any $x>0$ and any integer $k$ such that $x-k >0$, and in the end our assumption that $N\geq n^\eps$.	
	Still conditionally on $T_n$, we can perform a union-bound over all the $O(n^{4})$ possibilities for choosing the region $R$ and the two labels $i_1,i_2$ corresponding to the removed edges. Combining this with~\eqref{eqn:mass_region_and_W}, with high probability, for any region $\Rr$ and corresponding $I,i_1,i_2,N$, we have 
	\begin{align*}
		|\Rr| \geq \sum_{i\in I\setminus {i_1,i_2}} W_i \geq n^{-1-\eps}\cdot N \geq  n^{-1-\eps}\cdot  \# \tilde{R} \geq  n^{-1-\eps}\cdot  \# R ,
	\end{align*}
	which is what we wanted to prove.
\end{proof}
\subsection{The Aldous recursive decomposition}\label{subsec:aldous recursive decomposition}

\paragraph{Decomposing $\Tt$ into three regions.}
\begin{figure}
	\begin{tabular*}{15cm}{cc}
		\subfloat[A realization of the Brownian tree $\Tt$. \label{subfig:a}]{\includegraphics[page=1]{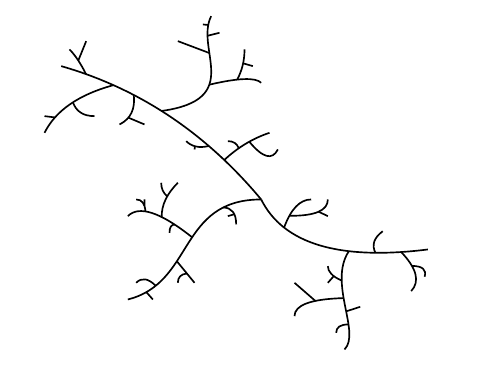}} &
		\subfloat[Its recursive decomposition up to depth $2$. \label{subfig:b}]{\includegraphics[page=3]{decomposition_into_regions}}
	\end{tabular*}
	\caption{A decomposed Brownian tree $\Tt$ and its decomposition  $(\Rr[\ii])_{\ii \in \mathbb{T}_3^2}$ of depth $2$.}\label{fig:recursive decomposition}
\end{figure}
\begin{figure}
	\begin{tabular*}{15cm}{cccc}
		\subfloat[A region delimited by one point. \label{subfig:a}]{\includegraphics[page=1]{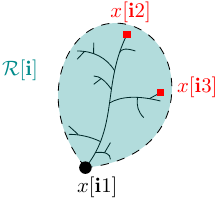}}&
		\subfloat[Its decomposition. \label{subfig:b}]{\includegraphics[page=2]{decomposition_region_1_point}}
		&
			\subfloat[A region delimited by two points. \label{subfig:b}]{\includegraphics[page=1]{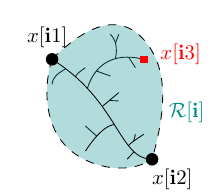}}
		&
			\subfloat[Its decomposition. \label{subfig:b}]{\includegraphics[page=2]{decomposition_region_2_points}}
	\end{tabular*}
	\caption{One step of the decomposition for a region $\Rr[\ii]$ delimited by respectively one point and two points. Note that in both cases, the newly created region $\Rr[\ii 3]$ is delimited by only one point.}\label{fig:one step recursive decomposition}
\end{figure}
We now introduce a recursive decomposition of the Brownian tree $\Tt$, which consists of repeatedly applying the above decomposition for $n=3$. 
More precisely, let $\left( \Tt, x[1], x[2], x[3] \right)$ be a $3$-pointed Brownian tree with mass $1$. 
Note that almost surely, the points $x[1]$, $x[2]$ and $x[3]$ are leaves. 
In this case, there exists a unique branching point $b[\emptyset]$ of $\Tt$ such that $x[1]$, $x[2]$ and $x[3]$ lie in three different connected components of $\Tt \backslash \{b[\emptyset]\}$. 
For $i \in \{1,2,3\}$, we call $\Rr[i]$ the (closure in $\Tt$ of the) connected component containing $x[i]$. 
Then Theorem~\ref{thm_discrete_continuous_coupling} for $n=3$ implies that the vector $\left( \left| \Rr[1] \right|, \left| \Rr[2] \right|, \left| \Rr[3] \right| \right)$ has distribution $\Dir \left( \frac{1}{2}, \frac{1}{2}, \frac{1}{2} \right)$ and that conditionally on this vector, the three regions $\left( \Rr[i], b[\emptyset], x[i] \right)$ are independent bi-pointed Brownian trees with prescribed masses. 
We will then pick a third point uniformly at random in each $\Rr[i]$ and apply this decomposition recursively.

\paragraph{The complete ternary tree.}
More precisely, let
\[ \mathbb{T}_3=\bigcup_{k \geq 0} \{1,2,3\}^k\]
be the set of finite words on the alphabet $\{1,2,3\}$. 
We will generally use the notation $\ii=i_1 i_2 \dots i_k$ to denote an element of $\mathbb{T}_3$. 
For such a word we call $k$ the \emph{depth} of $\ii$, and denote by $\mathbb{T}_3^k$ the set of elements of $ \mathbb{T}_3$ of depth $k$. 
The set $\mathbb{T}_3$ can be seen as the complete ternary tree, with root $\emptyset$ and where the parent of $i_1 i_2 \dots i_k$ is $i_1 i_2 \dots i_{k-1}$. 
For any $k \geq 0$ and any word $\ii \in \mathbb{T}_3$ with depth at least $k$, we will write $\ii_k=i_1 \dots i_k$.
Finally, we will use concatenation of words of $\mathbb{T}_3$: if $\ii=i_1 \dots i_k$ and $\mathbf{j}=j_1 \dots j_{\ell}$, we will write $\ii\mathbf{j}$ for $i_1 \dots i_k j_1 \dots j_{\ell}$.

\paragraph{The recursive decomposition of the Brownian tree.}

Our decomposition will associate to each word $\ii $ of $\mathbb{T}_3$ a region $\Rr[\ii ]$ of $\Tt$ in such a way that for all $\ii  \in \mathbb{T}_3$, the regions $\Rr[\ii 1]$, $\Rr[\ii 2]$ and $\Rr[\ii 3]$ form a partition of $\Rr[\ii ]$.
%(except for one branching point of $\Rr[\ii ]$).
More precisely, we write $\Rr[\emptyset]=\Tt$ and we define the branching point $b[\emptyset]$ and the bi-pointed trees $\Rr[1]$, $\Rr[2]$ and $\Rr[3]$ as in the previous paragraph. 
Moreover, let $\ii \in \mathbb{T}_3$ of depth $k \geq 1$ and assume that the region $\Rr[\ii ]$ has been defined, and call $x[\ii 1]$, $x[\ii 2]$ its marked points. 
Conditionally on all the rest, let $x[\ii 3]$ be a random point sampled according to the mass measure on $\Rr[\ii ]$. 
We denote by $b[\ii]$ the unique branching point of $\Rr[\ii ]$ such that $x[\ii 1]$, $x[\ii 2]$ and $x[\ii 3]$ lie in three different connected components of $\Rr[\ii ] \backslash \{b[\ii]\}$, and we call the closures of these three components respectively $\Rr[\ii 1]$, $\Rr[\ii 2]$ and $\Rr[\ii 3]$. 
Then, for $a \in \{1,2,3\}$, we set $x[\ii a1]:=b[\ii]$ and $x[\ii a 2]:=x[\ii a]$ (that is, we mark two points on each of the regions $\Rr[\ii a]$, in the same way as in the first step).
Finally, we equip the regions $\Rr[\ii a]$ with the metric and the measure naturally inherited from $\Rr[\ii ]$. 
See Figure~\ref{fig:recursive decomposition} and Figure~\ref{fig:one step recursive decomposition} for an illustration.

It can easily be checked by induction on the depth that this construction makes sense and that for all $\ii  \in \mathbb{T}_3$, the tree $\Rr[\ii]$ is a bi-pointed Brownian tree with a randomized mass. 
Indeed, if $\left( \Rr[\ii ], x[\ii 1], x[\ii 2] \right)$ is a Brownian tree, then almost surely the points $x[\ii 1]$, $x[\ii 2]$ and $x[\ii 3]$ defined above are pairwise distinct leaves, so $b[\ii]$ is uniquely characterized and distinct from $x[\ii 1]$, $x[\ii 2]$, $x[\ii 3]$ and Theorem~\ref{thm_discrete_continuous_coupling} for $n=3$ ensures that the $\Rr[\ii a]$ are Brownian trees. 
If $\ii $ has depth $k$, we will call $\Rr[\ii ]$ a \emph{region of scale $k$}. 
Using Theorem~\ref{thm_discrete_continuous_coupling} with $n=3$ at each scale, we easily have the following independence properties.

\begin{proposition}\label{prop_independence_decomposition}
The following statements hold. 
	\begin{enumerate}[(i)]
		\item The vectors 
		\begin{equation}\label{eqn:vector_relative_masses}
		\left( \frac{\left| \Rr[\ii 1] \right|}{\left| \Rr[\ii ] \right|}, \frac{\left| \Rr[\ii 2] \right|}{\left| \Rr[\ii ] \right|}, \frac{\left| \Rr[\ii 3] \right|}{\left| \Rr[\ii ] \right|} \right)
		\end{equation}
		for $\ii  \in \mathbb{T}_3$ are i.i.d.\ with distribution $\Dir \left( \frac12, \frac12, \frac12 \right)$.
		\item For all $k \geq 1$, conditionally on their masses, the measured metric spaces $\left( \Rr[\ii ], x[\ii 1], x[\ii 2] \right)$ for $\ii  \in \mathbb{T}_3^k$ are i.i.d.\ bi-pointed Brownian trees.
	\end{enumerate}
\end{proposition}

\begin{proof}
	Let $\mathcal{F}_k$ be the $\sigma$-algebra generated by the masses $\left| \Rr[\ii ] \right|$ for $\ii$ of depth at most $k$. 
	It is sufficient to show that for all $k \geq 0$, conditionally on $\mathcal{F}_k$, the vectors~\eqref{eqn:vector_relative_masses} for $\ii $ of depth $k$ are i.i.d.\ with law $\Dir \left( \frac12, \frac12, \frac12 \right)$, and that conditionally on $\mathcal{F}_k$ and on those vectors, the regions $\Rr[\ii ]$ for $\ii $ of depth $k+1$ are independent bi-pointed Brownian trees with prescribed masses.
	We prove this statement by induction on $k$.
	
	For $k=0$, this is just Theorem~\ref{thm_discrete_continuous_coupling} for $n=3$. If the statement is true for some $k \geq 0$, the induction hypothesis guarantees that conditionally on $\mathcal{F}_k$, the $\Rr[\ii ]$ for $\ii $ of depth $k+1$ are independent bi-pointed Brownian trees with randomized masses. We then apply Theorem~\ref{thm_discrete_continuous_coupling} for $n=3$ to each of these Brownian trees, still conditionally on $\mathcal{F}_k$.
\end{proof}

\paragraph{Zooming in on a random point.}
Finally, let $x \in \Tt$ and $k \geq 0$. 
If $x$ is not one of the points $x[\ii]$, which is the case for almost every $x \in \Tt$, we will denote by $\ii_k(x)$ the unique index $\ii  \in \mathbb{T}_3$ of depth $k$ such that $x \in \Rr[\ii ]$, and by $i_k(x)$ the $k$-th letter of the word $\ii _k(x)$. 
It will be useful for us to study the recursive decomposition around a random point $X$ picked uniformly in $\Tt$.  
In this setting, we will need the following result.

\begin{lemma}\label{lem_independence_pointed_decomposition}
	Let $\left( \Tt, \left( \Rr[\ii] \right)_{\ii \in \mathbb{T}_3} \right)$ be a decomposed Brownian tree with a point $X$ sampled uniformly on $\Tt$, independently of the decomposition.
	Then the vectors
	\begin{equation}\label{eqn:relative_masses_directed}
	\left( \frac{\left| \Rr[\ii _j(X)1] \right|}{\left| \Rr[\ii _j(X)] \right|}, \frac{\left| \Rr[\ii _j(X)2] \right|}{\left| \Rr[\ii _j(X)] \right|}, \frac{\left| \Rr[\ii _j(X)3] \right|}{\left| \Rr[\ii _j(X)] \right|}, i_{j+1}(X) \right)
	\end{equation}
	for $j \geq 0$ are i.i.d.. Moreover, they have the distribution of $(W_1, W_2, W_3, I)$, where $(W_1, W_2, W_3) \sim \Dir \left( \frac{1}{2}, \frac{1}{2}, \frac{1}{2} \right)$ and $\Ppsq{ I=i}{W_1, W_2, W_3}=W_i$ for $i \in \{1,2,3\}$. In particular, the index $I$ is uniform on $\{1,2,3\}$ and the variable $W_I$ follows the $\mathrm{Beta}\left( \frac{3}{2},1 \right)$ distribution.
%	$\frac{3}{2} \sqrt{x} \mathbbm{1}_{x \in [0,1]} \mathrm{d}x$.
\end{lemma}

\begin{proof}
	The argument is basically the same as the proof of Proposition~\ref{prop_independence_decomposition}, where $\mathcal{F}_k$ is replaced by the $\sigma$-algebra $\mathcal{G}_k$ generated by the vectors~\eqref{eqn:relative_masses_directed} for $0 \leq j \leq k-1$.	
	The only additional element is that since $X$ is picked independently of the decomposition, conditionally on the Brownian trees $\Rr[1]$, $\Rr[2]$,  $\Rr[3]$, it has probability $\left| \Rr[a] \right|$ to be in $\Rr[a]$ for $a \in \{1,2,3\}$. 
	In particular, we have
	\[ \Ppsq{i_1(X)=a}{\left| \Rr[1] \right|, \left| \Rr[2] \right|, \left| \Rr[3] \right|} = |\Rr[a]| \]
	and conditionally on $\left( \left| \Rr[1] \right|, \left| \Rr[2] \right|, \left| \Rr[3] \right|, i_1(X) \right)$, the region $\Rr[i_1(X)]$ is still a (bi-pointed) Brownian tree with prescribed mass. The heredity step is adapted in the same way.	
	Finally, the fact that $W_I$ has a $\mathrm{Beta}\left(\frac{3}{2},1\right)$ distribution is a consequence of~\eqref{eq:dirichlet distrib factorisation} and~\eqref{eq:dirichlet distrib size-biaising}.
\end{proof}
\begin{remark}\label{rem:mass in Tt as probability}
Often in the paper, it will be useful to re-interpret a quantity defined as the mass of a certain subset of $\Tt$, typically of the form $|\cup_{\ii \in U^k}\Rr[\ii]|$ with $U^k \subset \mathbb{T}_3^k$, as the probability that a random point $X$ falls into it. This will typically take the following form
\begin{align*}
	\left|\bigcup_{\ii \in U^k}\Rr[\ii]\right| = \Ppsq{\ii_k(X)\in U^k}{\Tt, \left( \Rr[\ii] \right)_{\ii \in \mathbb{T}_3} }.
\end{align*}
\end{remark}
\paragraph{Estimates for the size of the regions.}
We now present a very rough result controlling the sizes of the regions appearing in the Aldous recursive decomposition of $\Tt$.
\begin{lemma}\label{lem:rough_control_size_regions}
	There are constants $C>c>0$ such that almost surely, for $k$ large enough, for all $\ii  \in \mathbb{T}_3^k$, we have
	\[ e^{-Ck} \leq \left| \Rr[\ii ] \right| \leq e^{-ck}. \]
\end{lemma}

\begin{proof}
	The proof uses classical branching random walk arguments. 
	We first notice that, by the independence properties of our recursive decomposition, the process $\left( \log \left| \Rr[\ii ] \right| \right)_{\ii  \in \mathbb{T}_3}$ is a branching random walk. 
	That is, the vectors
	\[ \left( \log \left| \Rr[\ii 1] \right|-\log \left| \Rr[\ii ] \right|, \log \left| \Rr[\ii 2] \right|-\log \left| \Rr[\ii ] \right|, \log \left| \Rr[\ii 3] \right|-\log \left| \Rr[\ii ] \right| \right) \]
	for $\ii  \in \mathbb{T}_3$ are i.i.d.\ and have the law of $\left( \log W_1, \log W_2, \log W_3 \right)$, where $(W_1,W_2,W_3)$ has distribution $\Dir \left( \frac{1}{2}, \frac{1}{2}, \frac{1}{2} \right)$.
	Therefore, for any $C>0$, using the Chernoff bound we can write, for any $\lambda>0$, for any $\ii\in \mathbb{T}_3^k$,
	\begin{align}\label{eq:chernoff inequality upper bound size regions}
		\Pp{\log \left| \Rr[\ii ] \right| \leq -Ck} \leq e^{-\lambda C k} \Ec{e^{-\lambda \log(W_1)}}^k. 
	\end{align}
	Since Dirichlet distributions have polynomial tails, we can find $\lambda$ such that the expectation is finite. 
	We can then find $C$ such that the right-hand side in the last equation is bounded by $4^{-k}$. 
	We conclude the proof by a union bound over $\ii \in \mathbb{T}_3^k$ and then using the Borel--Cantelli lemma over $k$.
	
	Similarly, for $c>0$ and $\lambda>0$, we have
	\[ \Pp{\log \left| \Rr[\ii ] \right| \geq -ck} \leq e^{\lambda c k} \Ec{e^{\lambda \log(W_1)}}^k. \]
	Since $W_1$ does not have an atom at zero, there exists $\lambda>0$ such that $\Ec{e^{\lambda \log(W_1)}}<\frac{1}{5}$. 
	Once the value of $\lambda$ is fixed, we can choose $c>0$ sufficiently small so that the last right-hand side is bounded by $4^{-k}$. 
	This proves the other direction, again by union bound and Borel--Cantelli.
\end{proof}

\section{Proof of Proposition~\ref{prop_homeo_mass}} \label{sec:proof of prop homeo mass}
This entire section is devoted to proving Proposition~\ref{prop_homeo_mass}, which is central to the proof of the main results in the next section. 

\paragraph{Rough idea of the proof.}
We start with a rough idea of the proof. 
Let us fix a point $x \in \Tt$ and consider the ratios
\begin{equation}\label{eqn_ratios_around_x}
\left( \frac{\left| \Rr[\ijx] \right|}{\left| \Rr[\ii_{j-1}(x)] \right|} \right)_{1 \leq j \leq k}
\end{equation}
obtained by "zooming scale after scale" around $x$. By Lemma~\ref{lem_independence_pointed_decomposition}, these are i.i.d.\ variables with a fixed, absolutely continuous distribution. 
Using this, we will argue that if we fix a small region $r'$ of $\Tt'$, the probability that we can find nested regions of $\Tt'$ around $r'$ that respect the ratios~\eqref{eqn_ratios_around_x} up to a factor $1\pm\delta$ at most of the scales $1 \leq j \leq k$ is of order $\delta^k$. 
By choosing $\delta$ small enough, we will be able to do a union bound over the possible candidate regions $r'$ of $\Tt'$. 
This shows that for any small region $r'$ of $\mathcal{T}'$, a homeomorphism sending $x$ to a point of $r'$ will have "$\delta$-mismatches" between the ratios $\frac{\left| \Rr[\ijx] \right|}{\left| \Rr[\ii_{j-1}(x)] \right|}$ and $\frac{\left| \Psi \left( \Rr[\ijx] \right) \right|}{\left| \Psi \left( \Rr[\ii_{j-1}(x)] \right) \right|}$ for "many" scales $j$ (this is Proposition~\ref{prop_existence_mismatches} below, that we prove in Section~\ref{subsec:mismatches}). 
Moreover, the ratio 
\begin{equation}\label{eqn:ratio_rprime_vs_r}
	\frac{\left| \Psi \left( \Rr[\ikx] \right) \right|}{\left| \Rr[\ikx] \right|}
\end{equation}
can be written as the telescopic product of those mismatches, so the existence of mismatches proves that the ratio~\eqref{eqn:ratio_rprime_vs_r} will vary "quite often" by a factor $1\pm\delta$.
To conclude, we will argue that for most points $x$, those mismatches cannot compensate each other. 
Indeed, if for a region $\Rr[\ii]=\Rr[\ii 1] \cup \Rr[\ii 2] \cup \Rr[\ii 3]$ we have $\frac{\left|\Psi \left( \Rr[\ii] \right)\right|}{\left|\Rr[\ii] \right|} \ll 1$, then if a mismatch makes the ratio $\frac{\left|\Psi \left( \Rr[\ii 1] \right)\right|}{\left|\Rr[\ii 1] \right|}$ closer to $1$, it will make the ratios $\frac{\left|\Psi \left( \Rr[\ii 2] \right)\right|}{\left|\Rr[\ii 2] \right|}$ and $\frac{\left|\Psi \left( \Rr[\ii 3] \right)\right|}{\left|\Rr[\ii 3] \right|}$ even smaller. Another way to say this is that if $x$ is picked uniformly at random in $\Tt$, the ratios~\eqref{eqn:ratio_rprime_vs_r} form a nonnegative martingale in $k$ (see~\eqref{eqn:martingale} below), so they converge almost surely. By existence of mismatches, they change by a factor $1\pm\delta$ many times, so the limit of the martingale has to be $0$. This argument is done in a quantitative way in Section~\ref{subsec:martingale}.

\subsection{Finding mismatches}\label{subsec:mismatches}

Before giving the precise definition of mismatches, it will be convenient to restrict ourselves to a (not too small) subset of the intermediate scales. More precisely, we first notice that by construction of the Aldous recursive decomposition, if a word $\ii$ has its last letter equal to $3$, then the boundary of the region $\Rr[\ii]$ in $\Tt$ is a single point (see Figure~\ref{fig:one step recursive decomposition}).
\begin{definition}\label{defn:good_scale}
		Let $\alpha>0$ and let $\ii \in \mathbb{T}_3$ with depth $k$. 
		We say that a scale $1 \leq j \leq k-1$ is $\alpha$-\emph{good} for $\ii$ with respect to $\left( \Rr[\ii]\right)_{\ii \in \mathbb{T}_3}$ if $j$ is odd, if $i_j=i_{j+1}=3$, and if furthermore
	\begin{equation}\label{eqn:condition_alpha_good}
	\min \left( \left| \Rr[\ii_j 1] \right|, \left| \Rr[\ii_j 2] \right|, \left| \Rr[\ii_j 3] \right| \right) \geq \alpha \left| \Rr[\ii_j] \right|.
	\end{equation}
\end{definition}
The reason why we want $i_j=3$ is that it implies that there are two small regions $\Rr[\ii_j 2], \Rr[\ii_j 3]$ whose boundary is the singleton $\{ b[\ii_j] \}$ (see Figure~\ref{fig:one step recursive decomposition}), and the ratio between the masses of those regions gives a convenient way to say that $\Tt$ and $\Tt'$ have "different shapes" at a certain scale.
The point of the assumption~\eqref{eqn:condition_alpha_good} is that if the branching point $b[\ii]$ splits the region $\Rr[\ii]$ in a very uneven way, there will be many possible choices for $\Psi(b[\ii])$, whereas if $b[\ii]$ is "central" in $\Rr[\ii]$ only few choices will be available.

The first step is to make sure that around most points of the tree $\Tt$, good scales represent a positive proportion of the scales.
\begin{lemma}\label{lem_many_good_scales}
	Let $\left( \Tt, \left( \Rr[\ii] \right)_{\ii \in \mathbb{T}_3} \right)$ be a decomposed Brownian tree. 
	There exists constants $\alpha, \beta, \gamma>0$ such that denoting
	\begin{align*}
		U^k:= \enstq{\ii \in \mathbb{T}_3^k}{\ii \mbox{ has at least $\beta k$ scales that are $\alpha$-good}}
	\end{align*}
and $E^k:=\bigcup_{\ii \in \mathbb{T}_3^k \backslash U^k} \Rr[\ii]$, then almost surely, for $k$ large enough, we have
$|E^k| < e^{-\gamma k }$.
\end{lemma}
\begin{proof}
	%	This is a straightforward consequence of Lemma~\ref{lem_independence_pointed_decomposition}.
	Let us prove that we can find $\alpha,\beta,\gamma$ such that $\Ec{|E^k|}<e^{-2\gamma k}$, so that the conclusion of the lemma can be obtained using the Markov inequality and the Borel--Cantelli lemma. 
	Let $X$ be a uniform point taken under the mass measure on $\Tt$ independently of the decomposition.
	By Remark~\ref{rem:mass in Tt as probability}, we can rewrite $\Ec{|E^k|}$ in the following way:
	\begin{align*}
		\Ec{|E^k|}
		% = \Ec{\left|\enstq{x\in \Tt}{\ii(x)\in U^k}\right|}
		= \Pp{\ii_k(X) \mbox{ has less than $\beta k$ scales that are $\alpha$-good}}.
	\end{align*}
	 As in the proof of Lemma~\ref{lem_independence_pointed_decomposition}, for $k \geq 0$, we denote by $\mathcal{G}_k$ the $\sigma$-algebra generated by the word $\ii_k(X)$ and by the masses $\Rr[\ii_j (X) a]$ for $0 \leq j \leq k-1$ and $a \in \{1,2,3\}$. It is clear that the event that $j$ is a good scale for $\ii_k(X)$ belongs to $\mathcal{G}_{j+1}$.
	
	Moreover, by Lemma~\ref{lem_independence_pointed_decomposition}, for every odd $j$, we have
	\begin{align} \label{eq:probability that next scale is good}
	\Ppsq{\mbox{$j$ is $\alpha$-good for $X$}}{ \mathcal{G}_{j-1}} = \frac{1}{3} \cdot \frac{1}{3} \cdot \Pp{ \min \left( W_1, W_2, W_3 \right) \geq \alpha}, 
\end{align}
	where $(W_1,W_2,W_3) \sim \Dir \left( \frac12, \frac12, \frac12 \right)$.
	This probability is deterministic and does not depend on $j$ so we denote it by $p$, and note that $p>0$ if $\alpha>0$ was chosen small enough.	
	Therefore, the variables 
	\begin{align*}
		\left( \mathbbm{1}_{\{\mbox{the scale $2j+1$ is $\alpha$-good for $X$}\}} \right)_{0 \leq j < \frac{k}{2}}
	\end{align*}
	are i.i.d.\ Bernoulli variables with positive mean, which is sufficient to prove the lemma by a Chernoff bound.
\end{proof}
From now on, we will fix $\alpha, \beta, \gamma>0$ that satisfy the conclusion of Lemma~\ref{lem_many_good_scales} and we will simply refer to an $\alpha$-good scale as a \emph{good} scale.

Now let $0<\delta<\alpha$ and let $\Psi: \Tt \to \Tt'$ be a homeomorphism.
\begin{definition}\label{defn_mismatch}
Let $k \geq 0$ and $\ii \in \mathbb{T}_3^k$. We say that a scale $1 \leq j \leq k-1$ is a $\delta$-\emph{mismatch} for $\Psi$ at $\ii$ with respect to $\left( \Rr[\ii] \right)_{\ii \in \mathbb{T}_3}$ if it is a good scale and if we have
\[ \max_{1 \leq a \leq 3} \left| \frac{\left| \Rr[\ii_j a] \right|}{\left| \Rr[\ii_j] \right|}-\frac{\left| \Psi \left( \Rr[\ii_j a] \right) \right|}{\left| \Psi \left( \Rr[\ii_j] \right) \right|}  \right|>\delta. \]
\end{definition}
Since $\Tt$, $(\Rr[\ii])_{\ii \in \mathbb{T}_3}$ and $\delta$ are fixed throughout the paper, we will often simply write "mismatch for $\Psi$ at $\ii$" instead of "$\delta$-\emph{mismatch} for $\Psi$ at $\ii$ with respect to $\left( \Rr[\ii] \right)_{\ii \in \mathbb{T}_3}$".
Informally, the scale $j$ being a mismatch for $\Psi$ at $\ii$ indicates that when we perform one step of the Aldous recursive decomposition in $\Tt$, the decomposition of $\Rr[\ii_j]$ into three parts and its image by $\Psi$ in $\Psi \left( \Rr[\ii_j] \right)$ do not split the masses with the same proportions. 

Our next goal is the following result, which guarantees the existence of mismatches around most points for any homeomorphism. 
It represents a significant proportion of the proof of Proposition~\ref{prop_homeo_mass}.
\begin{proposition}\label{prop_existence_mismatches}
	There exists $\delta>0$ such that almost surely, for $k$ large enough, for any $\ii \in \mathbb{T}_3^k$ that has at least $\beta k$ good scales, for any homeomorphism $\Psi$ from $\Tt$ to $\Tt'$, one of the following holds:
	\begin{enumerate}
		\item \label{it:prop:image region is small} either $\left| \Psi \left( \Rr[\ii] \right) \right| \leq e^{-k} \left| \Rr[\ii] \right|$,
		\item\label{it:prop:many mismatches} there are at least $\lfloor\frac{\beta}{2} k \rfloor$ good scales in $\{1,\dots,k-1\}$ that are $\delta$-mismatches for $\Psi$ at $\ii$.
	\end{enumerate}
\end{proposition}
\begin{proof}
	Let $\delta \in \intervalleoo{0}{\alpha}$, whose value will be specified later. 
	Let $k \geq 0$. 
	Let $C>0$ be the constant given by the lower bound in Lemma~\ref{lem:rough_control_size_regions}, and let $B^k$ be the event that for all $\ii \in \mathbb{T}_3$ of depth at most $k$, the branching point $b[\ii]$ satisfies $\left| b[\ii] \right|_{\Tt} \geq e^{-Ck}$. 
	By Lemma~\ref{lem:rough_control_size_regions} (applied to $k+1$), almost surely $B^k$ occurs for $k$ large enough.
	
	Now let $\ii \in \mathbb{T}_3^k$, and let $J$ be a subset of $\{ 1,2,\dots, k-1\}$ with size $\lfloor \frac{\beta}{2} k \rfloor$ and with only odd elements. 
	We are interested in the existence of a homeomorphism $\Psi$ that would satisfy the three following properties:
	\begin{enumerate}[(i)]
		\item\label{it:homeo:large enough branching points}  $| b[\ii_j]|_{\Tt} \geq e^{-C k}$ for all $0 \leq j \leq k$,
		\item\label{it:homeo:image region large enough} we have $\left| \Psi \left( \Rr[\ii] \right) \right| > e^{-k} \left| \Rr[\ii] \right|$,
		\item\label{it:homeo:not a mismatch} for all $j \in J$, the scale $j$ is good but is not a $\delta$-mismatch for $\Psi$ at $\ii$.
	\end{enumerate}
We insist that in item (i), we mean $e^{-Ck}$ and not $e^{-Cj}$. We hence define the event $A_J(\ii)$ as
	\begin{align*}
		A_J(\ii) = \big\{ \mbox{there exists a homeomorphism $\Psi:\mathcal{T} \to \mathcal{T}'$ such that (\ref{it:homeo:large enough branching points}), (\ref{it:homeo:image region large enough}), (\ref{it:homeo:not a mismatch})  hold for $\Psi$ and $\ii$}\big\}.
	\end{align*}
In the rest of the proof, whenever there exists a homeomorphism $\Psi$ such that (\ref{it:homeo:large enough branching points}), (\ref{it:homeo:image region large enough}), (\ref{it:homeo:not a mismatch})  hold for $\Psi$ and $\ii$ we will say that "$\Psi$ makes $A_J(\ii)$ occur". 

Now let us consider what happens on the event where $B^k$ occurs but none of the $A_J(\ii)$ does for any $\ii \in \mathbb{T}_3^k$ and any $J\subset \{ 1,2,\dots, k-1\}$ with $\#J=\lfloor \frac{\beta}{2} k \rfloor$. 
Fix some $\Psi:\Tt\rightarrow \Tt'$ and $\ii \in \mathbb{T}_3^k$ that has at least $\beta k$ good scales.
On the event considered, (\ref{it:homeo:large enough branching points}) is satisfied so either (\ref{it:homeo:image region large enough}) fails (and in this case point \ref{it:prop:image region is small}. of the proposition holds); or (\ref{it:homeo:image region large enough}) holds for $\Psi$ and $\ii$, which entails that (\ref{it:homeo:not a mismatch}) fails for all choices of $J$. 
Since we assumed that $\ii$ has more than $\beta k$ good scales, the fact that we cannot find any  $\lfloor \frac{\beta}{2}k\rfloor$ good scales that are not $\delta$-mismatches tells us that at least $\beta k - \lfloor \frac{\beta}{2}k\rfloor\geq \lfloor\frac{\beta}{2}k\rfloor$ of them are indeed $\delta$-mismatches.  
In view of this, since we already know that $B^k$ occurs for $k$ large enough, it suffices to find $\delta>0$ such that almost surely, none of the $A_J(\ii)$ with $\ii \in \mathbb{T}_3^k$ and $J \subset \{1,\dots,k\}$ of cardinal $\lfloor \frac{\beta}{2}k \rfloor$ occurs for $k$ large enough. For that, we will show using a union bound that the probability of
	\begin{equation}\label{eqn_rewriting_event_Aij}
		\bigcup_{\ii \in \mathbb{T}_3^k} \bigcup_{\substack{J \subset \{1,\dots,k-1\}\\ \#J=\lfloor\frac{\beta}{2} k\rfloor }} A_J(\ii)
		\end{equation}
	is summable in $k$, and conclude with the Borel--Cantelli lemma.
	
	Suppose that there exists some $\Psi$ that makes the event $A_J(\ii)$ occur.	
	We write $\ell=\lfloor \frac{\beta}{2}k \rfloor$ and denote by $j_1 < j_2 < \dots < j_{\ell}$ the elements of $J$, and write $b_h=b[\ii_{j_{h}}]$ for all $1 \leq h \leq \ell$. 
	We now try to understand the sequence of points
	\[ \mathbf{c}=(c_h)_{1 \leq h \leq \ell} = \left( \Psi(b_h) \right)_{1 \leq h \leq \ell}.\]
	We first note that this sequence must satisfy a topological condition: we recall that $i_{j_h}=i_{j_h+1}=3$ because the scale $j_h$ is good. Hence, for all $1 \leq h \leq \ell$, the point $b_h$ is a branching point of $\mathcal{T}$ with the points $b_1, \dots, b_{h-1}$ in the same connected component of $\mathcal{T} \backslash \{b_h\}$, and the points $b_{h+1}, \dots, b_{\ell}$ in another component. 
	Since $\Psi$ is a homeomorphism, the same is true for the sequence $(c_h)$ in $\mathcal{T}'$. Moreover, we claim that the sizes of the branching points $c_h$ cannot be too small. 
	More precisely, for all $1 \leq h \leq \ell-1$, let $\Rr'_h(\mathbf{c})$ be the connected component of $\mathcal{T}' \backslash \{ c_h\}$ that contains $c_{h+1}, \dots, c_{\ell}$, with the convention that $\Rr'_0(\mathbf{c})=\mathcal{T}'$. 
	In other words, we have $\Rr'_h(\mathbf{c})=\Psi \left( \Rr[\ii_{j_h+1}] \right)$. 
	Then for all $1 \leq h \leq \ell$, we have 
	\[ |c_h|_{\Rr'_{h-1}(\mathbf{c})} = |c_h|_{\Psi \left( \Rr[\ii_{j_{h-1}+1}]
	 \right)} \geq |c_h|_{\Psi \left( \Rr[\ii_{j_{h}}] \right)} = \min_{a \in \{1,2,3\}} \left| \Psi \left( \Rr[\ii_{j_h} a] \right) \right|.\]
	Using successively the fact that $j_h$ is not a $\delta$-mismatch and the fact that the scale $j_h$ is $\alpha$-good, we deduce, for $1 \leq h \leq \ell-1$:
	\begin{equation}\label{eqn_ci_not_too_small}
	|c_h|_{\Rr'_{h-1}(\mathbf{c})} \geq \left( \frac{\min_{a \in \{1,2,3\}} \left| \Rr[\ii_{j_h} a] \right| }{\left| \Rr[\ii_{j_h}] \right|} -\delta \right) \cdot \left| \Psi \left( \Rr[\ii_{j_h}] \right) \right| \geq (\alpha-\delta) \left| \Rr'_h(\mathbf{c}) \right|.
	\end{equation}
	From now on, we assume $\delta<\frac{\alpha}{4}$, so that
	\begin{align*}
	|c_h|_{\Rr'_{h-1}(\mathbf{c})} \geq \frac{\alpha}{2} \left| \Rr'_h(\mathbf{c}) \right|
	\end{align*}
	for $1 \leq h \leq \ell-1$. 
	For the same reasons, we also have
	\begin{align*}
		|c_{\ell}|_{\Rr'_{\ell-1}(\mathbf{c})} \geq \frac{\alpha}{2} \left| \Psi \left( \Rr[\ii_{j_{\ell}}]\right) \right| \geq \frac{\alpha}{2} \left| \Psi \left( \Rr[\ii] \right) \right| 
	\end{align*}
	so, using (\ref{it:homeo:image region large enough}) and (\ref{it:homeo:large enough branching points}) in the definition of the event $A_J(\ii)$, we have
	\begin{align*}
	|c_{\ell}|_{\Rr'_{\ell-1}(\mathbf{c})} \geq \frac{\alpha}{2} e^{-k} \left| \Rr[\ii] \right| \geq \frac{\alpha}{2} e^{-(C+1)k}. 
	\end{align*}
	Following what precedes, we define a \emph{candidate sequence of length $\ell$} as a sequence $(c_h)_{1 \leq h \leq \ell}$ of branching points of $\mathcal{T}'$ such that:
	\begin{itemize}
		\item for all $1 \leq h \leq \ell$, the points $c_1, \dots, c_{h-1}$ all lie in one connected component of $\mathcal{T}' \backslash \{c_h\}$, and the points $c_{h+1}, \dots, c_{\ell}$ lie in another one, denoted if $h<\ell$ by $\Rr'_h(\mathbf{c})$;
		\item for all $1 \leq h \leq \ell-1$, we have $|c_h|_{\Rr'_{h-1}(\mathbf{c})} \geq \frac{\alpha}{2} \left| \Rr'_{h}(\mathbf{c}) \right|$;
		\item we have $|c_{\ell}|_{\Rr'_{\ell-1}(\mathbf{c})} \geq \frac{\alpha}{2}e^{-(C+1)k}$.
	\end{itemize}
Note that crucially, the notion of candidate sequence does not depend on $\delta$. By the discussion above, any $\Psi:\Tt\rightarrow \Tt'$ that makes the event $A_J(\ii)$ occur must send $(b_h)_{1\leq h \leq \ell}$ to some candidate sequence of length $\ell$ in $\Tt'$.
	We will now provide a bound on the number of such candidate sequences in $\Tt'$. 
	This bound is entirely deterministic, and only uses the fact that $\mathcal{T}'$ is a real tree with total mass $1$ where all branching points have degree $3$, which is almost surely the case for a realization of a Brownian tree.
	\begin{lemma}\label{lem_counting_candidates}
		There exists a constant $K=K(C,\alpha, \beta)$ such that almost surely, for all $\ell \geq 2$, the number of candidate sequences of length $\ell$ in $\mathcal{T}'$ is at most $K^{\ell}$.
	\end{lemma}
\begin{figure}
	\centering
	\includegraphics[page=1]{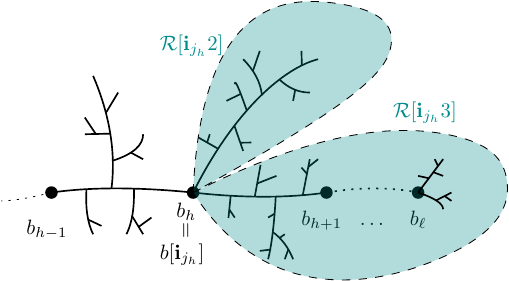}
	\caption{The sequence $(b_h)_{1\leq h \leq \ell}$ and their positions relative to $\Rr[\ii_{j_h}2]$ and $\Rr[\ii_{j_h}3]$.}
	\label{fig:region containing end of branchpoint sequence}
\end{figure}
	\begin{proof}
		If $\mathbf{c}=(c_h)_{1 \leq h \leq \ell}$ is a candidate sequence, we first define its \emph{sequence of scales} $\left( s_h(\mathbf{c}) \right)_{0 \leq h \leq \ell}$ by
		\[ s_h(\mathbf{c}) = \lfloor -\log \left| \Rr'_h(\mathbf{c}) \right| \rfloor \]
		for $1 \leq h \leq \ell-1$, with the conventions $s_0(\mathbf{c})=0$ and $s_{\ell}(\mathbf{c})=\lfloor -\log |c_{\ell}|_{\Rr'_{\ell-1}(\mathbf{c})} \rfloor$. 
		Then $\mathbf{s}(\mathbf{c})=\left( s_h(\mathbf{c}) \right)_{1 \leq h \leq \ell}$ is a non-decreasing sequence of integers starting at $0$.
		Moreover, by definition of a candidate sequence, the numbers $s_h(\mathbf{c})$ are bounded above by
		\[ -\log \left( \frac{\alpha}{2} e^{-(C+1)k}\right) = -\log \frac{\alpha}{2} + (C+1) k.\]
		Therefore, the number of possible values of the sequence $\left( s_h(\mathbf{c}) \right)_{0 \leq h \leq \ell}$ is at most
		\begin{equation}\label{eqn_number_scale_sequences}
			\binom{\log (2/\alpha) + (C+1)k + \ell}{\ell} \leq \frac{2}{\alpha} 2^{(C+2)k}.
		\end{equation}
		On the other hand, for any such sequence $\mathbf{s}=(s_h)_{0 \leq h \leq \ell}$, let us bound the number of candidate sequences $\mathbf{c}$ for which the scale sequence $\mathbf{s}(\mathbf{c})$ is $\mathbf{s}$. 
		For this, we start with an easy remark. 
		Let $\Rr$ be a region of $\mathcal{T}$ of mass $m_0$. 
		Then we claim the following
		\begin{claim*}
		The number of branching points $c$ in $\Rr$ satisfying $|c|_{\Rr} \geq m$ is at most $\frac{m_0}{m}$. 
		\end{claim*}
		Indeed, let $S_m$ be the set of branching points of size at least $m$ in $\Rr$. 
		Then $\Rr$ can be obtained by gluing the connected components of $\Rr \backslash S_m$ along the structure of a finite binary tree. 
		The nodes of this binary tree correspond to the points of $S_m$, so this binary tree has $\# S_m$ nodes and therefore $\# S_m+2$ leaves. 
		Those $\# S_m+2$ leaves correspond to $\# S_m+2$ disjoint parts of $\Rr$ with mass at least $m$ each, so $\left( \# S_m+2 \right) m \leq m_0$ and the claim follows.
		
		Now, let $\mathbf{s}=(s_h)_{0 \leq h \leq \ell}$ be a non-decreasing sequence of integers with $s_0=0$ and $s_{\ell} \leq \log \frac{2}{\alpha} +(C+1)\ell$. 
		Let us build step by step a candidate sequence $\mathbf{c}$ satisfying $\mathbf{s}(\mathbf{c})=\mathbf{s}$:
		\begin{itemize}
			\item For $\mathbf{c}$ to be a candidate sequence, $c_1$ needs to satisfy 
			\begin{align*}
				|c_1|_{\Rr'_{0}(\mathbf{c})}=|c_1|_{\Tt} \geq \frac{\alpha}{2}\cdot |\Rr'_{1}(\mathbf{c})| \geq \frac{\alpha}{2}\cdot e^{-(s_1(\mathbf{c})+1)}
			\end{align*}
			by definition of $s_1(\mathbf{c})$. 
			Using the last display and the claim, the number of possible choices for $c_1$ for which $s_1(\mathbf{c})=s_1$ is at most $\frac{2}{\alpha} e^{s_1+1}$. 
			\item Let $1 \leq h \leq \ell-1$ and assume that $c_1, \dots, c_{h-1}$ have already been chosen. 
			Then $\Rr'_{h-2}(\mathbf{c})$ is determined by $(c_1, \dots, c_{h-1})$ and $\Rr'_{h-1}(\mathbf{c})$ is a connected component of $\Rr'_{h-2}(\mathbf{c}) \backslash \{ c_{h-1} \}$, so there are only $3$ possible choices for $\Rr'_{h-1}(\mathbf{c})$.
			Moreover, once this region has been chosen, the point $c_h$ must be a point of $\Rr'_{h-1}(\mathbf{c})$ with 
			\[
			|c_h|_{\Rr'_{h-1}(\mathbf{c})} \geq \frac{\alpha}{2} \left| \Rr_h'(\mathbf{c}) \right| \geq \frac{\alpha}{2} e^{-(s_h(\mathbf{c})+1)} 
			\]
			by definition of $s_h(\mathbf{c})$. 
			Using the claim again and the fact that $|\Rr'_{h-1}(\mathbf{c})| \leq e^{-s_{h-1}(\mathbf{c})}=e^{-s_{h-1}}$ by construction, 
			the number of possible choices for $c_h$ that ensure that $s_h(\mathbf{c})=s_h$, given $c_1, \dots, c_{h-1}$, is bounded above by
			\[ 
			3 \cdot \frac{\left| \Rr'_{h-1}(\mathbf{c}) \right|}{|c_h|_{\Rr'_{h-1}(\mathbf{c})}} \leq \frac{6}{\alpha} e^{s_h-s_{h-1}+1}. 
			\]
			\item Finally, by the same reasoning, the number of possible choices for $c_{\ell}$ given $c_1, \dots, c_{\ell-1}$ is bounded above by
			\[ \frac{6}{\alpha} e^{(C+1)\ell-s_{\ell-1}}.\]
		\end{itemize}
		Using the above in cascade and reducing the telescopic product we get that, for any $\mathbf{s}$, the number of candidate sequences $\mathbf{c}$ such that $\mathbf{s}(\mathbf{c})=\mathbf{s}$ is bounded above by
		\[
		 \left( \frac{6}{\alpha} \right)^{\ell} e^{(C+2)\ell}.
		 \]
		Combined with~\eqref{eqn_number_scale_sequences}, this proves the lemma, with
		\begin{align}\label{eq:formula for K} 
		K=\frac{6}{\alpha} \cdot e^{C+2} \cdot  2^{\frac{2C+4}{\beta}}.
		\end{align}
	\end{proof}

		We return to the proof of Proposition~\ref{prop_existence_mismatches}.
		From now on, we will work conditionally on the tree $\mathcal{T'}$ and fix $J,\ii$. 
		We recall that $j_1<\dots<j_{\ell}$ are the elements of $J$ and that $b_h=b[\ii_{j_h}]$. 
		We have seen that if there exists a $\Psi$ that makes the event $A_J(\ii)$ occur, then the sequence $\left( \Psi(b_h) \right)_{1 \leq h \leq \ell}$ is a candidate sequence in $\mathcal{T}'$. 
		Therefore, we fix such a candidate sequence $(c_h)_{1 \leq h \leq \ell}$, and estimate the probability that there exists a $\Psi$ satisfying (\ref{it:homeo:large enough branching points}), (\ref{it:homeo:image region large enough}), (\ref{it:homeo:not a mismatch})  as well as $\Psi(b_h)=c_h$ for all $1 \leq h \leq \ell$. 
		For this, let $2 \leq h \leq \ell-1$. 
		On the event that such a $\Psi$ exists, since $j_h$ is not a $\delta$-mismatch for $\ii$, we have
		\begin{equation}\label{eqn:absence_mismatch}
		\left( \frac{ \left| \Psi \left( \Rr[\ii_{j_h}2] \right) \right|}{\left| \Psi \left( \Rr[\ii_{j_h}] \right) \right|}-\delta \right) \left|\Rr[\ii_{j_h}] \right| \leq \left| \Rr[\ii_{j_h}2] \right| \leq \left( \frac{ \left| \Psi \left( \Rr[\ii_{j_h}2] \right) \right|}{\left| \Psi \left( \Rr[\ii_{j_h}] \right) \right|}+\delta \right) \left|\Rr[\ii_{j_h}] \right|,
		\end{equation}
		and a similar estimate holds for $\left| \Rr[\ii_{j_h}3] \right|$. 
		
		On the other hand, by our conventions in the construction of the recursive decomposition and the fact that $i_{j_h}=i_{j_h+1}=3$ (since $j_h$ is a good scale for $\ii$), the connected component of $\Tt \backslash\{ b_h \}$ that contains $b_{\ell}$ is $\Rr[\ii_{j_h}3]$, and the component that contains neither $b_1$ nor $b_{\ell}$ is $\Rr[\ii_{j_h}2]$, see Figure~\ref{fig:region containing end of branchpoint sequence}.
		Hence $\Psi \left( \Rr[\ii_{j_h}2] \right)$ must be the connected component of $\Tt' \backslash \{c_h\}$ that contains neither $c_1$ nor $c_{\ell}$, and $\Psi \left( \Rr[\ii_{j_h}3] \right)$ must be the connected component of $\Tt' \backslash \{c_h\}$ that contains $c_{\ell}$. 
		We denote by $n^h_2:=|\Psi \left( \Rr[\ii_{j_h}2] \right)|$ and $n^h_3:=|\Psi \left( \Rr[\ii_{j_h}3] \right)|$ the respective masses of these two connected components, and highlight that those masses are completely determined by the sequence $\mathbf{c}$. 
		From~\eqref{eqn:absence_mismatch} and the analog equation for $\Rr[\ii_{j_h}3]$, using that $(x,y) \mapsto \frac{x}{x+y}$ is increasing in $x$ and decreasing in $y$, we get		
		\begin{equation}\label{eqn:absence_mismatch_23}
		\frac{n^{h}_2 - \delta \left| \Psi \left( \Rr[\ii_{j_h}] \right) \right|}{n^{h}_2+n^{h}_3} \leq \frac{\left| \Rr[\ii_{j_h}2] \right|}{\left| \Rr[\ii_{j_h}2] \right|+\left| \Rr[\ii_{j_h}3] \right|} \leq \frac{n^{h}_2+\delta \left| \Psi \left( \Rr[\ii_{j_h}] \right) \right|}{n^{h}_2+n^{h}_3}.
		\end{equation}
		Since the scale $j_h$ is $\alpha$-good and we have chosen $\delta<\frac{\alpha}{2}$, we have as in~\eqref{eqn_ci_not_too_small} that $n^{h}_2, n^h_3 \geq \frac{\alpha}{2} \left| \Psi \left( \Rr[\ii_{j_h}] \right) \right|$, so $n_2^h+n_3^h \geq \alpha \left| \Psi \left( \Rr[\ii_{j_h}] \right) \right|$ and~\eqref{eqn:absence_mismatch_23} becomes
		\[ \left| \frac{\left| \Rr[\ii_{j_h}2] \right|}{\left| \Rr[\ii_{j_h}2] \right|+\left| \Rr[\ii_{j_h}3] \right|}-\frac{n^{h}_2}{n^{h}_2+n^{h}_3} \right| \leq \frac{\delta}{\alpha} \]
		for all $2\leq h \leq  \ell-1$. On the other hand, by Proposition~\ref{prop_independence_decomposition}, the variables $\frac{\left| \Rr[\ii_{j}2] \right|}{\left| \Rr[\ii_{j}2] \right|+\left| \Rr[\ii_{j}3] \right|}$ for $j \geq 0$ are i.i.d.\ copies of $\frac{W_2}{W_2+W_3}$, where $(W_1, W_2, W_3) \sim \Dir \left( \frac12, \frac12, \frac12 \right)$. 
		It follows that we have
		\[ \Ppsq{\exists \Psi \text{  that makes  $A_J(\ii)$ occur and }\forall \, 2 \leq h \leq \ell-1, \Psi(b_h)=c_h}{ \mathcal{T'}} \leq \prod_{h=2}^{\ell-1} \Pp{ \left| \frac{W_2}{W_2+W_3}-q_h \right| \leq \frac{\delta}{\alpha} },  \]
		where $q_h=\frac{n^{h}_2}{n^{h}_2+n^{h}_3} \in \left[ \frac{\alpha}{2}, 1-\frac{\alpha}{2} \right]$. 
		On the other hand, by \eqref{eq:dirichlet distrib factorisation}, the law of $\frac{W_2}{W_2+W_3}$ is $\mathrm{Beta}\left(\frac{1}{2},\frac{1}{2}\right)$
		with density
		\(\frac{1}{\pi} \frac{1}{\sqrt{t(1-t)}} \mathbbm{1}_{[0,1]}(t) \, \mathrm{d}t.\)
		In particular, the density is bounded on $\left[ \frac{\alpha}{4}, 1-\frac{\alpha}{4} \right]$ by $\frac{1}{\pi \sqrt{ \frac{1}{2} \cdot \frac{\alpha}{4}}}\leq \frac{1}{\sqrt{\alpha}}$. 
		Therefore, integrating between $q_h-\frac{\delta}{\alpha}$ and $q_h+\frac{\delta}{\alpha}$, assuming that $\delta<\frac{\alpha^2}{4}$ we have
		\[ 
		\Ppsq{\exists \Psi \text{  that makes  $A_J(\ii)$ occur and }\forall \, 2 \leq h \leq \ell-1, \Psi(b_h)=c_h}{ \mathcal{T'}} \leq \left( \frac{2\delta}{\alpha\sqrt{\alpha}} \right)^{\ell-2}.\]
		We can now take the union bound over all candidate sequences $(c_h)$. By Lemma~\ref{lem_counting_candidates}, we obtain
		\[ \Ppsq{A_J(\ii)}{ \mathcal{T'}} \leq K^{\ell} \left( \frac{2\delta}{\alpha\sqrt{\alpha}} \right)^{\ell-2}.\]
		We can now remove the conditioning on $\mathcal{T'}$ and perform a union bound over the $3^k$ possible values of $\ii$ and the $\binom{k-1}{\ell} \leq 2^k$ possible values of the set $J$. 
		We find
		\begin{align}\label{eq:how small delta has to be}
				\Pp{\bigcup_{\ii\in \mathbb{T}_3^k} \bigcup_{\substack{J \subset \{1,\dots,k-1\}\\ \#J=\ell}} A_J(\ii)} \leq 6 ^k \cdot \left( \frac{2\delta K}{\alpha\sqrt{\alpha}} \right)^{\ell-2}.  
		\end{align} 
		Reminding that $\ell=\lfloor\frac{\beta k}{2}\rfloor$, if we choose the constant $\delta>0$ sufficiently small, this decays exponentially in $k$. 
		Given the discussion before \eqref{eqn_rewriting_event_Aij}, this proves the proposition.
\end{proof}

\subsection{The martingale argument}\label{subsec:martingale}
For the next part of the argument, we need to introduce another notion of mismatch that is slightly looser than the one of Definition~\ref{defn_mismatch}. 
From now on, we fix a value of $\delta>0$ that satisfies the conclusion of Proposition~\ref{prop_existence_mismatches}.
Suppose that $\left( \Tt, \left( \Rr[\ii] \right)_{\ii \in \mathbb T_3} \right)$ is a decomposed Brownian tree and $\Tt'$ is another Brownian tree, and that $\Psi:\Tt\rightarrow \Tt'$ is a homeomorphism.
\begin{definition}\label{defn:weak_mismatch}
	Let $j \geq 1$, and let $\ii \in \mathbb{T}_3$ of depth at least $j+1$. We say that the scale $j$ is a weak mismatch for $\Psi$ at $\ii$ if $\frac{\left| \Rr[\ii_j a] \right|}{\left| \Rr[\ii_j] \right|}\geq \alpha$ for all $a \in \{1,2,3\}$ and furthermore
	\[ \max_{a\in \{1,2,3\}} \left| \frac{|\Psi(\Rr[\ii_j a])|}{|\Psi(\Rr[\ii_j])|} - \frac{|\Rr[\ii_j a]|}{|\Rr[\ii_j]|} \right| > \delta. \]
\end{definition}

In particular, a mismatch as defined in Definition~\ref{defn_mismatch} is also a weak mismatch. 
The difference between the two definitions is that we have removed the "topological" part of the assumption that $j$ has to be a good scale, \emph{i.e.} that $i_j=i_{j+1}=3$ (see Definition~\ref{defn:good_scale}).
In this section, we prove the following result.
\begin{proposition}\label{prop:mass of bad regions is small}
Let $\beta$ be as in Lemma~\ref{lem_many_good_scales}. 
There exists a constant $\eta>0$ such that for any homeomorphism $\Psi$ from $\Tt$ to $\Tt'$, for any $k\geq 1$, the set
\begin{align*}
	V^{k,\Psi}:=\enstq{\ii \in \mathbb{T}_3^k}{\Psi \text{ has more than $\frac{\beta}{2} k$ weak mismatches at $\ii$ and }\frac{\left| \Psi \left( \Rr[\ii] \right) \right|}{\left| \Rr[\ii] \right|}  \geq e^{-\eta k}}
\end{align*}
is such that 
\begin{align*}
	\left|\bigcup_{\ii \in V^{k,\Psi}}\Rr[\ii]\right| < e^{-\eta k}.
\end{align*}
\end{proposition}
We highlight that the result is in fact deterministic: it holds almost surely for realizations of Brownian trees $\Tt, \Tt'$, decomposition $(\Rr[\ii])_{\ii\in \mathbb{T}_3^k}$ and any homeomorphism $\Psi$.
In what follows, we will refer to weak mismatches as simply mismatches.
\begin{proof}
Although the result is deterministic, we will give a probabilistic proof
by interpreting the mass of a subset of $\Tt$ as the probability that a point $X$ sampled uniformly in $\Tt$ belongs to that subset, as explained in Remark~\ref{rem:mass in Tt as probability}.
In all the proof, we treat $\Tt, \Tt'$ as deterministic, compact real trees, equipped with a nonatomic mass measure, and we also treat $\left( \Rr[\ii] \right)_{\ii \in \mathbb{T}_3}$ as deterministic. 
We pick $X$ in $\Tt$ according to its mass measure. 
For $j \geq 0$, we denote by $\mathcal{H}_j$ the $\sigma$-algebra generated by $\ii_j(X)$, so that $(\mathcal H_j)_{j\geq 0}$ is a filtration. 
Note that since $X$ is uniform, we have for $j \geq 0$ and $a \in \{1,2,3\}$
\begin{equation}\label{eqn_law_j+1}
	\Ppsq{i_{j+1}(X)=a}{\mathcal H_j} = \frac{\left| \Rr[\ii_j(X) a] \right|}{\left| \Rr[\ii_j(X)] \right|}.
\end{equation}
We also note that if $0 \leq j \leq k$, then the event that scale $j$ is a mismatch for $\Psi$ at $\ii_k(X)$ is $\mathcal{H}_j$-measurable\footnote{This is the reason why we are looking at the weak mismatches of Definition~\ref{defn:weak_mismatch}, and not at the mismatches of Definition~\ref{defn_mismatch}: the assumption $i_{j+1}=3$ in the definition of a good scale is not $\mathcal{H}_j$-measurable.}, since it only depends on $\ii_j(X)$.
Now, for $j \geq 0$, we define 
\begin{equation}\label{eqn:martingale}
M_j:=\frac{\left| \Psi \left( \Rr[\ii_j(X)] \right) \right|}{\left| \Rr[\ii_j(X)] \right|}.
\end{equation}
A simple computation using~\eqref{eqn_law_j+1} shows that the process $(M_j)_{j \geq 0}$ is an $(\mathcal H_j)$-martingale.

The idea behind Proposition~\ref{prop:mass of bad regions is small} is that a mismatch gives an opportunity for $M_{j+1}$ to be significantly different from $M_j$. Since the martingale $M$ is positive, it converges almost surely, and if its value changes often the limit has to be $0$.
To obtain a quantitative version of this intuition, we will study $(\log M_j)_{j \geq 0}$, which is a supermartingale. We will use the fact that the steps corresponding to mismatches tend to bring the value of this process down by more than an additive constant in expectation. Hence, after a large number $k$ of steps, either we have seen few mismatches or $\log M$ has gone down by a lot.

More precisely, let us fix a constant $\mu>0$ (to be precised later). For $j \geq 1$, we introduce
\begin{align*}
 Z_j := \log M_j - \log M_{j-1} \quad \text{and} \quad \widetilde{Z}_j=  Z_j + \mu\cdot \mathbbm{1}_{\{\text{scale $j-1$ is a mismatch for $\ii_k(X)$}\}}.
\end{align*}
We will prove that if $\mu>0$ is chosen sufficiently small, then for all $j\geq 1$, we have
\begin{align}\label{eq:small exponential moment tilde Zk}
 \Ecsq{\exp\left(\frac12\widetilde{Z}_j\right)}{\mathcal H_{j-1}} \leq 1.
\end{align}
From here, the result follows from using a Chernoff bound. Indeed, using the last display in cascade, we obtain $\Ec{\exp \left(\frac{1}{2} \sum_{j=1}^{k}\widetilde{Z}_j \right)}\leq 1$ so that we have
\begin{align*}
	 \Pp{\sum_{j=1}^{k}\widetilde{Z}_j \geq 2 \eta k}\leq \exp(-\eta k).
\end{align*}
Writing $\sum_{j=1}^{k}\widetilde{Z}_j= \log M_k + \mu\cdot \sum_{j=1}^{k}\mathbbm{1}_{\{\text{scale $j-1$ is a mismatch for $\ii_k(X)$}\}}$, we obtain 
\begin{align*}
	\Pp{\log M_k \geq (2\eta -\mu \beta/2)k \text{ and } \ii_k(X) \text{ has more than $\frac\beta2 k$ mismatches}} \leq \exp(-\eta k),
\end{align*}
which is what we want to prove if we set $\eta=\frac{\mu\beta}{8}$.

So now, we only have to prove \eqref{eq:small exponential moment tilde Zk} for some value $\mu>0$. 
First, on the event that scale $j-1$ is not a mismatch, we have
\begin{align*}
	\Ecsq{e^{\widetilde Z_j}}{\mathcal H_{j-1}}= \Ecsq{e^{Z_j}}{\mathcal H_{j-1}}= \Ecsq{\frac{M_j}{M_{j-1}}}{\mathcal H_{j-1}}=1
\end{align*}  
by the martingale property. Therefore, by concavity of $x\mapsto x^\frac{1}{2}$ and Jensen's inequality, we have $\Ecsq{e^{\frac12 \widetilde Z_j}}{\mathcal H_{j-1}}\leq 1^\frac{1}{2} =1$. 
Hence we only have to focus on the event where scale $j-1$ is a mismatch. 
For this, we use the following lemma.
\begin{lemma}\label{lem:small exponential moment for Z+c}
There exists $\mu=\mu(\alpha,\delta)$ such that the following holds. 
Let $(p_1,p_2,p_3)$ and $(q_1,q_2,q_3)$ be two elements of the simplex $\enstq{(x_1,x_2,x_3)}{x_1+x_2+x_3=1}$, and let $Z$ be a random variable given by
\begin{align*}
 Z= \log q_I - \log p_I,\qquad \text{where for all } i\in \{1,2,3\}, \quad \Pp{I=i}=p_i.
\end{align*}
If $\max_{1 \leq i \leq 3}|p_i-q_i| \geq \delta$ and $p_1,p_2,p_3\geq \alpha$,
then we have
\begin{align}
 \Ec{\exp\left(\frac{1}{2} (Z + \mu )\right)}\leq 1.
\end{align}
\end{lemma}
We then apply the lemma to the random variable $\widetilde{Z}_j$ conditionally on $\mathcal{H}_{j-1}$, on the event that $j-1$ is a mismatch, by taking $p_i= \frac{\left| \Rr[\ii_{j-1}(X)i] \right|}{\left| \Rr[\ii_{j-1}(X)] \right|}$ and $q_i= \frac{\left| \Psi \left( \Rr[\ii_{j-1}(X)i] \right) \right|}{\left| \Psi \left( \Rr[\ii_{j-1}(X)] \right) \right|}$.
This ensures that \eqref{eq:small exponential moment tilde Zk} is satisfied for the appropriate choice of $\mu$ given in the lemma.
\end{proof}

\begin{proof}[Proof of Lemma~\ref{lem:small exponential moment for Z+c}]
For $\mu>0$, we have
		\[ \Ec{\exp \left( \frac12(Z+\mu)\right)} = e^{\frac12 \mu} \sum_{i=1}^3 p_i^{\frac12}\cdot q_i^{\frac12}. \]
		For each term, we have $p_i^{\frac12} \cdot q_i^{\frac12} \leq \frac{p_i+q_i}{2}$. Moreover, let $a$ be such that $|p_a-q_a| \geq \delta$. We have
		\[ p_a^{\frac12} \cdot q_a^{\frac12} = \frac{p_a+q_a}{2} - \frac{1}{2} \left( p_a^{\frac12}-q_a^{\frac12} \right)^2 \leq \frac{p_a+q_a}{2} - \frac{1}{2} \left( \frac{\delta}{2} \right)^2,\]
	since $\sqrt{y}-\sqrt{x} \geq \frac{1}{2}(y-x)$ for $0 \leq x \leq y \leq 1$. Hence, we have
	\[ \Ec{\exp \left( \frac12(Z+\mu)\right)} = e^{\frac12 \mu} \sum_{i=1}^3 p_i^{\frac12} \cdot q_i^{\frac12} \leq e^{\frac12 \mu} \left( -\frac{\delta^2}{8} + \sum_{i=1}^3  \frac{p_i+q_i}{2}  \right) = e^{\frac12 \mu} \left( 1-\frac{\delta^2}{8} \right). \]
	This proves our claim, by taking $\mu>0$ small enough.
\end{proof}

Given Propositions~\ref{prop_existence_mismatches} and~\ref{prop:mass of bad regions is small}, it is now easy to conclude the proof of Proposition~\ref{prop_homeo_mass}.
\begin{proof}[Proof of Proposition~\ref{prop_homeo_mass}]
We know that almost surely, for $k$ large enough, the conclusions of Lemma~\ref{lem_many_good_scales} and Proposition~\ref{prop_existence_mismatches} hold. We fix $k$ for which it is the case.
Let $\Psi: \Tt \to \Tt'$ and consider some $\ii \in U^{k,\Psi}$, meaning that $\left| \Psi \left( \Rr[\ii] \right) \right| > e^{-\eta k} \left| \Rr[\ii] \right|$. 
Then
\begin{itemize}
	\item either $\ii \in U^k$ as defined by Lemma~\ref{lem_many_good_scales}, meaning that $\ii$ has less than $\beta k$ good scales,
	\item or Item 2 of Proposition~\ref{prop_existence_mismatches} holds, as Item 1 is prohibited since $\left| \Psi \left( \Rr[\ii] \right) \right| > e^{-\eta k} \left| \Rr[\ii] \right|$, so the region $\ii$ must have at least $\lfloor \frac{\beta}{2}k \rfloor$ scales that are $\delta$-mismatches, hence also weak $\delta$-mismatches. This entails, using our initial assumption on $\ii$, that $\ii \in V^{k,\Psi}$. 
\end{itemize}
Therefore, on the event that we considered, we have $U^{k,\Psi} \subset U^k \cup V^{k,\Psi}$, so
\begin{align*}
	\left|\bigcup_{\ii \in U^{k,\Psi}} \Rr[\ii]\right| < e^{-\eta k} + e^{-\gamma k} \leq e^{-\xi k},
\end{align*}
where $\gamma$ and $\eta$ are given respectively by Lemma~\ref{lem_many_good_scales} and Proposition~\ref{prop:mass of bad regions is small}. This concludes the proof by taking $0<\xi<\min(\gamma,\eta)$.
\end{proof}

\section{Proofs of the main results}\label{sec:proof of main results}

\subsection{H\"older homeomorphisms}
We start with the proof of Theorem~\ref{thm_holder_homeo}, which follows quite straightforwardly from Proposition~\ref{prop_homeo_mass}.
\begin{proof}[Proof of Theorem~\ref{thm_holder_homeo}] 
	Let $\left( \Tt, \left( \Rr[\ii] \right)_{\ii \in \mathbb{T}_3} \right)$ be a decomposed Brownian tree and let $\Tt'$ be an independent Brownian tree. Let also $\Psi:\Tt\rightarrow\Tt'$ be a homeomorphism. We will find a constant $\zeta>0$ such that almost surely, for $k$ large enough, we can find $\ii \in \mathbb{T}_3^k$ such that
	\begin{align}
		\diam(\Psi(\Rr[\ii ]))^{1-\zeta}\leq  \diam(\Rr[\ii ]).
	\end{align}
	Since the maximal diameter over all the regions of level $k$ tends to $0$ as $k \to +\infty$, this will entail that $\Psi^{-1}$ cannot be $(1-\frac{\zeta}{2})$-Hölder. 
	Since the problem is symmetric in $\Tt$ and $\Tt'$, this also shows that $\Psi$ cannot be $(1-\frac{\zeta}{2})$-Hölder either.
	 
	Let $k \geq 0$. On the one hand, we know that almost surely, for $k$ large enough, the conclusions of Proposition~\ref{prop_homeo_mass} and Lemma~\ref{lem:rough_control_size_regions} hold. On the other hand, we recall that by Proposition~\ref{prop_independence_decomposition}, conditionally on their masses, the regions $(\Rr[\ii ])_{\ii \in 
		\mathbb{T}_3^k}$ are independent Brownian trees with those respective masses. 
	Moreover, there exists a constant $u>0$ such that the Brownian tree $\Tt$ of mass $1$ satisfies $\Pp{\diam(\Tt)\leq x} \leq \exp(-ux^{-2})$ for all $x>0$ (this can e.g. be deduced from the explicit distribution of the maximum~\cite{kennedy_distribution_1976}). By union bound and Borel--Cantelli, there is a constant $c>0$ such that almost surely, we have for $k$ large enough
	\begin{align*}
		\min_{\ii \in \mathbb{T}_3^k} \left\lbrace\frac{\diam (\Rr[\ii ])}{\sqrt{|\Rr[\ii ]|}} \right\rbrace \geq \frac{c}{\sqrt{k}}.
	\end{align*}
	Combining this with the upper bound of Lemma~\ref{lem:rough_control_size_regions} (which ensures that $|\Rr[\ii]|$ behaves roughly as a decreasing exponential in $k$), 
	this implies that for any $\eps>0$, almost surely for $k$ large enough and $\ii \in \mathbb{T}_3^k$, we have
	\begin{equation}\label{eq:diam is at least mass to the half minus epsilon}
		\diam (\Rr[\ii ]) \geq c \sqrt{\frac{|\Rr[\ii ]|}{k}} \geq |\Rr[\ii ]|^{\frac{1}{2-\eps}}.
	\end{equation}
	
	Controlling the diameter of the regions $\Psi \left( \Rr[\ii]\right)$ of $\Tt'$ cannot be done in the same way, as we do not have a priori estimates on the shape of $\Psi \left( \Rr[\ii]\right)$. Therefore, we will use the definition of $\Tt'$ via the Brownian excursion, and the fact the excursion is H\"older.
	More precisely, we recall from Section~\ref{subsec:BCRT} that $\Tt'$ is built as a quotient of $[0,1]$ using a Brownian excursion $\mathbf{e}'=\left( e_t' \right)_{0 \leq t \leq 1}$. We denote by $p_{\mathbf{e}'}$ the canonical projection from $[0,1]$ to $\Tt'$.
	Now, for any $\ii \in \mathbb{T}_3$, the region $\Rr[\ii]$ is delimited by at most $2$ points. 
	Therefore, the same is true for the region $\Psi \left( \Rr[\ii] \right)$ of $\Tt'$. 
	This implies that $\Psi \left( \Rr[\ii] \right)$ is of the form $p_{\mathbf{e}'}(I_1 \cup I_2 \cup I_3)$, where $I_1 \cup I_2 \cup I_3$ is the union of three sub-intervals of $[0,1]$. 
	By connectedness of $\Psi \left( \Rr[\ii] \right)$, we have
	\[ \diam \left( \Psi \left( \Rr[\ii] \right) \right) \leq \sum_{j=1}^3 \diam \left( p_{\mathbf{e'}}(I_j) \right). \]
	Now let $\eps>0$. We know that if $|s-t|$ is small enough, then $|e'_s-e'_t| \leq \frac13 |s-t|^{\frac12-\eps}$. On the other hand, we know that $\left| \Psi \left( \Rr[\ii] \right)\right| $ goes a.s. to $0$ as the depth of $\ii$ goes to $+\infty$, so almost surely, for $k$ large enough and $\ii$ of depth $k$, we can write
	\begin{equation}
		\label{eq:diam of region is smaller than mass to some power}
		\diam \left( \Psi \left( \Rr[\ii] \right) \right) \leq \sum_{j=1}^3 \diam \left( p_{\mathbf{e'}}(I_j) \right) \leq \sum_{j=1}^3 \frac13 |I_j|^{\frac12-\eps} \leq \left| \Psi \left( \Rr[\ii] \right) \right|^{\frac12-\eps},
	\end{equation}
	since $p_{\mathbf{e}'}$ is measure-preserving.
	
	We can finally put things together. 
	Using the conclusion of Proposition~\ref{prop_homeo_mass}, almost surely for $k$ large enough, there exists $\ii \in \mathbb{T}_3^k$ such that
	\begin{equation}
		\label{eq:mass of Psi of a region is smaller than mass to some power}
		\left| \Psi \left( \Rr[\ii] \right) \right| \leq e^{-\eta k} \left| \Rr[\ii] \right| \leq \left| \Rr[\ii] \right|^{1+\eta/C},
	\end{equation}
	where the second inequality comes from Lemma~\ref{lem:rough_control_size_regions}.
	For this $\ii$, we have
	\begin{align*}
		\diam(\Psi(\Rr[\ii ])) \underset{\eqref{eq:diam of region is smaller than mass to some power}}{\leq} |\Psi(\Rr[\ii ])|^{\frac{1}{2}-\eps} \underset{\eqref{eq:mass of Psi of a region is smaller than mass to some power}}{\leq} |\Rr[\ii ]|^{(\frac{1}{2}-\eps)(1+\frac{\eta}{C})} \underset{\eqref{eq:diam is at least mass to the half minus epsilon}}{\leq} \diam(\Rr[\ii ])^{(\frac{1}{2}-\eps)(2-\eps)(1+\frac{\eta}{C})}.
	\end{align*}
	From there, we just need to take $\eps>0$ small enough to conclude. This proves the theorem and we can take $\eps_{\ref{thm_holder_homeo}}<\frac{\eta}{2C}$.
\end{proof}
\subsection{Maximum agreement subtree bound}
In this section, we prove Theorem~\ref{thm_MAST_upper} from two intermediate results. 
The first one, Corollary~\ref{prop:sum sqrt prod mass}, is a direct corollary of Proposition~\ref{prop_homeo_mass}. The second one, Lemma~\ref{lem:bound mast of two regions}, roughly says that the square root upper bound holds simultaneously in all regions of $T_n$ and $T'_n$ (its proof relies on the same ideas as the classic square root upper bound).
We will only state this lemma here, and prove it in the next subsection.
\begin{corollary}\label{prop:sum sqrt prod mass}
	Let $\left( \mathcal{T}, \left( \Rr[\ii] \right)_{\ii \in \mathbb{T}_3} \right)$ be a decomposed Brownian tree and let $\mathcal{T}'$ be an independent Brownian tree. There exists a constant $\rho>0$ such that with probability $1-o_k(1)$, for any homeomorphism $\Psi: \mathcal T \rightarrow \mathcal{T}'$, we have
	\begin{align*}
		\sum_{\ii \in \mathbb{T}_3^k} \sqrt{|\Rr[\ii] | \cdot |\Psi(\Rr[\ii] )|} \leq \exp(- \rho \cdot k).
	\end{align*}
\end{corollary}
\begin{proof}
	Let $\xi, \eta>0$ be given by Proposition~\ref{prop_homeo_mass}, and let $\Psi : \Tt \to \Tt'$ be a homeomorphism. 
	For $U^{k,\Psi}$ defined as in Proposition~\ref{prop_homeo_mass} and on the event of probability $1-o_k(1)$ on which its conclusion holds, we can write for $k$ large enough
	\begin{align*}
		\sum_{\ii \in \mathbb{T}_3^k} \sqrt{|\Rr[\ii] | \cdot |\Psi(\Rr[\ii] )|}
		&= \sum_{\ii \in \mathbb{T}_3^k\setminus U^{k,\Psi}} \sqrt{|\Rr[\ii] | \cdot |\Psi(\Rr[\ii] )|} + \sum_{\ii \in U^{k,\Psi}} \sqrt{|\Rr[\ii] | \cdot |\Psi(\Rr[\ii] )|}\\
		&\leq \sum_{\ii \in \mathbb{T}_3^k\setminus U^{k,\Psi}} |\Rr[\ii] |  \cdot \sqrt{ \frac{|\Psi(\Rr[\ii] )|}{|\Rr[\ii] |}} + \sqrt{\sum_{\ii \in U^{k,\Psi}} |\Rr[\ii] | }\cdot \sqrt{\sum_{\ii \in U^{k,\Psi}} |\Psi(\Rr[\ii] )|}\\
		&\leq 1 \cdot e^{-\eta k/2} + e^{-\xi k/2} \cdot 1 \\
		&< e^{-\rho k},
	\end{align*}
 where $\rho$ is chosen so that $\rho<\frac{1}{2}\min(\eta, \xi)$, the first inequality follows from the Cauchy--Schwarz inequality, and the second from Proposition~\ref{prop_homeo_mass}.	
\end{proof} 

\begin{lemma}\label{lem:bound mast of two regions}
	Let $\eps>0$. With high probability as $n\rightarrow \infty$, for any two regions $R\subset T_n$ and $R'\subset T_n'$, we have
	\begin{align}\label{eqn_general_sqrt_bound}
	\mathrm{MAST}(R,R') \leq 4e\sqrt{2} \cdot \left( n^\eps \vee \sqrt{\frac{(\#R)\cdot (\#R')}{n}}\right).
	\end{align}
\end{lemma}

We can now prove Theorem~\ref{thm_MAST_upper}. 
\begin{proof}[Proof of Theorem~\ref{thm_MAST_upper}]
	We let $\theta= \min \left( \frac{1}{2C}, \frac{1}{4 \log 3} \right)$, where $C>0$ is given by Lemma~\ref{lem:rough_control_size_regions} and we take $k=\lfloor \theta \log n \rfloor$. Note that in particular, the number of regions of the recursive decomposition of $\Tt$ at scale $k$ is $3^k\leq n^{\theta \log 3} \leq n^{\frac{1}{4}}$.
	We now let $\eps =\frac{\theta \rho}{4}$, where $\rho$ is given by Corollary~\ref{prop:sum sqrt prod mass}. We assume that $T'_n$ is coupled with an $n$-pointed Brownian tree $\left( \Tt', (X'_j)_{1 \leq j \leq n} \right)$ in the way described in Section~\ref{subsec:discrete_continuous_coupling}. 
	We also assume that $\left(\Tt', (X'_j) \right)$ is independent from $\Tt$, $\left( \Rr[\ii] \right)_{\ii \in \mathbb{T}_3^k}$ and $T_n$.
		
	Suppose that for some $S\subset \{1,\dots, n\}$, we have $T_n\vert_S=T_n'\vert_S=t$. Then we claim that there exists a homeomorphism $\Psi$ from $\Tt$ to $\Tt'$ such that $\Psi(X_j)=X'_j$ for all $i \in S$. 
	Indeed, up to reparametrization of the edges, there exists a unique embedding $\varphi$ of $t$ into $\Tt$ that sends the leaf labelled $j$ to $X_j$ for any label $j$ appearing in $t$, and a unique embedding $\varphi'$ of $t$ into $\Tt'$ that sends similarly $j$ to $X'_j$. For every edge $e=(x_e,y_e)$ of $t$, let $\mathcal{T}_e$ be the set of points $x \in \mathcal{T}$ such that the closest point of $\varphi(t)$ to $x$ belongs to $\varphi(e)$. 
	We define similarly the region $\mathcal{T}'_e \subset \Tt'$. 
	For every $e$, the regions $\mathcal{T}_e$ and $\mathcal{T}'_e$ are compact real trees where branching points are dense and all have degree $3$, so they have the same topology by~\cite[Theorem 1]{BT21} (see also~\cite{CH08})\footnote{The results in~\cite{BT21} are only stated for unpointed Brownian trees, but the proofs extend straightforwardly to bipointed trees.}. 
	Therefore, there exists a homeomorphism $\Psi_e : \mathcal{T}_e \to \mathcal{T}'_e$ such that $\Psi_e \left( \varphi(x_e) \right)=\varphi'(x_e)$ and $\Psi_e \left( \varphi(y_e) \right)=\varphi'(y_e)$. 
	The homeomorphism $\Psi$ is obtained by patching the $\Psi_e$ together. 
	Finally, for any $\ii \in \mathbb{T}_3^k$, we denote by $R[\ii]$ the smallest region of $T_n$ that contains all the labels $j \in \{1,\dots,n\}$ such that $X_j \in \Rr[\ii]$. 
	%We note that $R[\ii]$ is a region of $T_n$ (recall from Section~\ref{subsec_discrete_defns} that a region is either empty, or a single leaf, or delimited by at most two nodes). 
	We define similarly the regions $R'[\ii]$ of $T'_n$ using $\Tt'$ and $(X'_j)$.
	
	For every $\ii \in \mathbb{T}_3^k$, note that by definition of $R[\ii], R'[\ii]$ and by the fact that $\Psi(X_j)=X'_j$ for $j \in S$, we have $S \cap R[\ii]=S \cap R'[\ii]$. In particular, this subset induces the same subtree in $R[\ii]$ and in $R'[\ii]$. Therefore, we can write
	\[ \# S = \sum_{\ii \in \mathbb{T}_3^k} \#(S \cap R[\ii]) \leq \sum_{\ii \in \mathbb{T}_3^k} \MAST \left( R[\ii], R'[\ii] \right). \]
	On the other hand, Lemma~\ref{lem:bound mast of two regions} ensures that with probability $1-o_n(1)$, for all $\ii \in \mathbb{T}_3^k$ we have
	\begin{align*}
		\mathrm{MAST}(R[\ii] ,R'[\ii]) 
		&\leq 24e \cdot \left(n^\eps \vee \sqrt{\frac{(\#R[\ii] )(\#R'[\ii] )}{n}}\right)\\ 
		&\leq 24e \left(n^\eps + \sqrt{\frac{(\#R[\ii] )(\#R'[\ii] )}{n}}\indicator{\{\#R[\ii]  \geq  n^\eps \text{ and } \#R'[\ii] \geq n^\eps\}} \right).
	\end{align*}
	We now use Lemma~\ref{lem:comparison mass vs leaf-count}: with probability $1-o_n(1)$, for any $\ii \in \mathbb{T}_3^k$ such that $\#R[\ii], \#R'[\ii] \geq n^{\eps}$, we have $\#R[\ii]\leq n^{1+\eps}\cdot|\Rr[\ii]|$ and $\#R'[\ii] \leq  n^{1+\eps}\cdot|\Psi(\Rr[\ii])|$, so 
\begin{align*}
	\sqrt{\frac{(\#R[\ii] )(\#R'[\ii])}{n}} \leq n^{\frac{1}{2}+\eps} \cdot \sqrt{|\Rr[\ii] | \cdot |\Psi(\Rr[\ii] )|}.
\end{align*}
Finally putting everything together, we get 
	\begin{align*}
		\#S & \leq   24 e \cdot \left( 3^k\cdot n^\eps + n^{\frac{1}{2}+\eps}\cdot \sum_{\ii \in \mathbb{T}_3^k} \sqrt{|\Rr[\ii] | \cdot |\Psi(\Rr[\ii])|}\right)\\
		&\underset{\text{Cor.~\ref{prop:sum sqrt prod mass}}}{=} O(n^{\frac{1}{4}+\eps}) +  O(n^{\frac{1}{2}+\eps -\rho \theta})\\
		&\underset{\text{choice of $\eps$}}{<} n^{\frac{1}{2}-\frac{\rho \theta}{2}},
	\end{align*}
with probability $1-o_n(1)$.
\end{proof}
\subsection{The refined square root bound}
This section is devoted to proving Lemma~\ref{lem:bound mast of two regions}. 
Before getting to the proof, we will first state and prove two other intermediate results, Lemma~\ref{lem:square root bounds on size of mast} and Lemma~\ref{lem:size of the intersection of two indep subsets}.

First, Lemma~\ref{lem:square root bounds on size of mast}, stated below, is a consequence of \cite[Lemma~4.1]{bryant_size_2003}, see also \cite[Lemma~4.1, Proposition~4.2]{bernstein_bounds_2015}, but expressed in a slightly different context.
We provide a quick proof for completeness, adapted from the same references. We say that a random labelled tree is \emph{exchangeable} if its distribution is invariant under uniform random permutation of the labels on the leaves.
\begin{lemma}\label{lem:square root bounds on size of mast}
	Suppose that $S_m$ and $S_m'$ are independent exchangeable random variables on $\mathcal{B}_m$ with $m$ leaves (with possibly different distribution). Then for any $s\geq 1$,
	\begin{align*}
	\Pp{\mathrm{MAST}(S_m,S_m') \geq s} \leq \binom{m}{s} \frac{2^{s-2}s}{s!}.
	\end{align*}
	In particular
	\begin{align*}
	\Pp{\mathrm{MAST}(S_m,S_m') \geq 2 e \sqrt{m}} \leq \exp \left(-2e (\log 2) \sqrt{m} + o(\sqrt{m}) \right)
	\end{align*}
as $m \rightarrow \infty$.
\end{lemma}
\begin{proof}
We use a first moment method to write
\begin{align*}
\Pp{\mathrm{MAST}(S_m,S_m')\geq s} &\leq \Ec{\sum_{\substack{A\subset \{1,\dots,m\}\\ \#A=s}} \mathbbm{1}_{S_m\vert_A=S_m'\vert_A}} \\
&=\binom{m}{s} \Pp{(S_m)\vert_{\{1,\dots,s\}}=(S_m')\vert_{\{1,\dots,s\}}},
\end{align*}
where the last equality follows from the exchangeability of the leaf-labels.
Now, we can bound the probability appearing on the right-hand-side of the last display by
\begin{align*}
\Pp{(S_m)\vert_{\{1,\dots,s\}}=(S_m')\vert_{\{1,\dots,s\}}} 
\leq \sup_{t\in \mathcal{B}_s} \Pp{(S_m)\vert_{\{1,\dots,s\}}=t}. 
\end{align*}
For $t,t'\in \mathcal{B}_s$ we write $t\sim t'$ if $t'$ can be obtained from $t$ by relabeling its leaves.
Note that this defines an equivalence relation on $\mathcal{B}_s$. 
Then for any random variable $T_s$ on $\mathcal{B}_s$ that satisfy the exchangeability property, we have 
\begin{align*}
	\Pp{T_s=t}= \Pp{T_s\sim t} \cdot \Ppsq{T_s=t}{T_s\sim t}
	\leq \Ppsq{T_s=t}{T_s\sim t}
	&=\frac{1}{\#\enstq{t'}{t'\sim t}}.
\end{align*}
Since the distribution of $(S_m)\vert_{\{1,\dots,s\}}$ satisfies the exchangeability condition, we just need to check that the number appearing at the denominator on the right-hand-side of the inequality is bounded from below by $\frac{s!}{2^{s-2}s}$ for any $t\in \mathcal{B}_s$.
For that, it suffices to prove that the number of graph automorphisms of any tree $t\in \mathcal{B}_s$ is bounded below by $2^{s-2}s$. This follows from the fact that an automorphism of a tree is determined by the image of one leaf and a cyclic ordering of the edges around each node. This proves the first claim of the lemma, and the second follows by the Stirling formula.
\end{proof}
In the proof of Lemma~\ref{lem:bound mast of two regions}, the above result will be used to control the size of the MAST of two regions $R\subset T_n$ and $R'\subset T_n$ in terms of their number of common labels.
The goal of the next result is to bound the number of these common labels in terms of $m=\# R$ and $m'=\# R'$.  
\begin{lemma}\label{lem:size of the intersection of two indep subsets}
Let $\eps>0$. 
Let $S$ and $S'$ be two independent uniform random subsets of $\{1,\dots,n\}$ of respective sizes $m,m'$. Then we have the following bound:
\begin{align*}
	\Pp{\# (S\cap S') \geq 8 \cdot \left(n^\eps\vee\frac{m m'}{n}\right) } \leq 2\exp \left(-\frac{2n^\eps}{3}\right).
\end{align*}  
\end{lemma}
\begin{proof}
We write $\# (S\cap S')$ as a sum of indicators $\# (S\cap S')=\sum_{i=1}^{n}\mathbbm{1}_{i\in S\cap S'}$. We also denote by $(\mathcal H_i)_{0\leq i \leq n}$ the filtration generated by the sequence $(\mathbbm{1}_{i\in S},\mathbbm{1}_{i\in S'})_{1\leq i \leq n}$. We can check that
\begin{align*}
	\Ppsq{i+1\in S\cap S'}{\mathcal H_{i}}= \frac{m-\sum_{j=1}^{i}\mathbbm{1}_{i\in S}}{n-i} \cdot \frac{m'-\sum_{j=1}^{i}\mathbbm{1}_{i\in S'}}{n-i} \leq \frac{mm'}{(n-i)^2}.
\end{align*}
Therefore, we have the following stochastic domination 
\begin{align*}
	\sum_{i=1}^{\frac{n}{2}} \mathbbm{1}_{i\in S\cap S'} \preceq_{\mathrm{st}} \mathrm{Bin}\left(\frac{n}{2},1 \wedge \frac{4mm'}{n^2}\right)\preceq_{\mathrm{st}}  \mathrm{Bin}\left(\frac{n}{2},1\wedge 4\cdot \left(n^{\eps -1} \vee \frac{mm'}{n^2}\right) \right).
\end{align*}
By shuffling the indices, the same domination holds for $\sum_{i=1}^{\frac{n}{2}} \mathbbm{1}_{i\in S\cap S'}$, so we obtain
\begin{align*}
	&\Pp{\# (S\cap S') \geq  8 \cdot \left(n^\eps\vee\frac{m m'}{n}\right)}\\
	 &\leq \Pp{\sum_{i=1}^{\frac{n}{2}} \mathbbm{1}_{i\in S\cap S'}\geq  4 \cdot \left(n^\eps\vee\frac{m m'}{n}\right)} 
	 +\Pp{\sum_{i=\frac{n}{2}+1}^n \mathbbm{1}_{i\in S\cap S'}\geq  4 \cdot \left(n^\eps\vee\frac{m m'}{n}\right)} \\
	&\leq 2 \exp\left(-\frac{2n^\eps}{3}\right),
\end{align*}
where for the last inequality we use that for a binomial random variable $X$ with expectation $\mu$, we have $\Pp{X\geq (1+\eps)\mu}\leq \exp(\frac{\eps \mu }{3})$.
\end{proof}
We can finally prove Lemma~\ref{lem:bound mast of two regions}.
\begin{proof}[Proof of Lemma~\ref{lem:bound mast of two regions}]
We will actually prove a much stronger statement, which is that the estimate of Lemma~\ref{lem:bound mast of two regions} holds even if we condition on the shape of $T_n$, $T'_n$.

More precisely, let $\sigma,\sigma'$ be two independent uniform random permutations of $\{1,\dots,n\}$, defined on the same probability space and independent of $T_n, T_n'$. We denote by $T_n^\sigma$ (resp. $(T_n')^{\sigma'}$) the tree obtained from $T_n$ (resp. $T'_n$) by replacing each label $j$ be the label $\sigma(j)$ (resp. $\sigma'(j)$). By exchangeability of the model, the couple $(T_n^\sigma,(T_n')^{\sigma'})$ has the same distribution as $(T_n,T_n')$ so we can prove the lemma for the former, and it will hold for the latter as well. For any region $R$ (resp. $R'$) of $T_n$ (resp. $T'_n$), we denote by $R^{\sigma}$ (resp. $(R')^{\sigma'}$) the corresponding region in $T_n^{\sigma}$ (resp. $(T'_n)^{\sigma'}$). That is, the label $j$ is in $R^{\sigma}$ if and only if $\sigma^{-1}(j)$ is in $R$.

Note that if one of our regions is empty or consists of a single leaf, the result is obvious, so we may focus on regions delimited by $1$ or $2$ nodes. 
Let $E_n$ be the event that there exist two regions $R \subset T_n$ and $R' \subset T'_n$ delimited by at most two nodes such that~\eqref{eqn_general_sqrt_bound} fails for the regions $R^{\sigma}$ and $(R')^{\sigma'}$. 
We want to show that $\Pp{E_n} \to 0$ as $n \to +\infty$. 
We will actually show that this is true even if we condition on $T_n, T'_n$, that is
\begin{align}\label{eqn:sqrt_bound_conditioned_on_shape}
	\max_{t,t' \in \mathcal{B}_n} \Pp{E_n \, | \, (T_n, T'_n)=(t,t')} \xrightarrow[n \to +\infty]{} 0.
\end{align}
For this, we fix $t,t' \in \mathcal{B}_n$. 
From now on, we condition on $(T_n,T'_n)=(t,t')$. 
Let $r,r'$ be two regions of $t,t'$ delimited by at most $2$ nodes each. 
Since the number of nodes in those two trees is fixed and equal to $n-2$, there are $O(n^4)$ such pairs $(r,r')$.
We denote by $m,m'$ the respective number of leaves of $r$ and $r'$.
We denote by $\mathrm{Lab}(\sigma,r)$ (resp. $\mathrm{Lab}(\sigma',r')$) the set of labels contained in the region $r^{\sigma}$ (resp. $(r')^{\sigma'}$) of $T_n^{\sigma}$ (resp. $(T'_n)^{\sigma'}$). 
Then $\mathrm{Lab}(\sigma,r)$ and $\mathrm{Lab}(\sigma',r')$ are two independent random uniform subsets of $\{1,\dots,n\}$ of respective size $m$ and $m'$. 

Any common induced subtree to $r^\sigma$ and $(r')^{\sigma'}$ can only use leaf-labels that are common to those two regions, so denoting $L=\mathrm{Lab}(\sigma,r)\cap \mathrm{Lab}(\sigma',r')$ we have
\begin{align*}
	\mathrm{MAST}(r^\sigma, (r')^{\sigma'})=\mathrm{MAST}(r^\sigma{|_L}, (r')^{\sigma'}{|_L}).
\end{align*}
Moreover, conditionally on $(T_n, T'_n,L)$, the random labelled trees $r^\sigma{|_L}$ and $(r')^{\sigma'}{|_L}$ are independent (maybe with different distributions) and have exchangeable leaf-labels taken from $L$.

Piecing things together, we get
\begin{align*}
	&\Ppsq{\mathrm{MAST}(r^\sigma, (r')^{\sigma'})\geq 4e\sqrt{2} \left(  n^\eps \vee \sqrt{\frac{m m'}{n}} \right)}{(T_n,T'_n)=(t,t')}\\
	&\leq \Ppsq{\#L \geq 8 \cdot \left(\frac{m m'}{n} \vee n^{2\eps} \right)}{ (T_n,T'_n)=(t,t')} \\
	&+ \Ppsq{\mathrm{MAST}(r^{\sigma}{|_L}, (r')^{\sigma'}{|_L})\geq 4e\sqrt{2} \cdot  \sqrt{n^{2\eps} \vee\frac{mm'}{n}} \mbox{ and } \#L \leq 8 \cdot \left(\frac{m m'}{n} \vee n^{2\eps} \right) }{ (T_n,T'_n)=(t,t')}\\
	&\leq 2 \exp\left(-\frac{2n^{2\eps}}{3}\right) + \exp\left(-4e\sqrt{2}(\log 2) n^\eps + o(n^\eps) \right).
\end{align*}
Where we have bounded the first term using Lemma~\ref{lem:size of the intersection of two indep subsets}, and the second by Lemma~\ref{lem:square root bounds on size of mast} (which applies by exchangeability) applied to $8 \left( \frac{mm'}{n} \wedge n^{2\eps} \right)$. Moreover, those bounds do not depend on $(t,t')$ and on $(r,r')$, so we can sum over the $O(n^4)$ choices of $(r,r')$ to obtain~\eqref{eqn:sqrt_bound_conditioned_on_shape}.
\end{proof}

\section{Discussion of the results} \label{sec:discussion of the results}

\subsection{Explicit constants}\label{subsec:explicit constants}

The goal of this paragraph is to obtain a quantitative lower bound for the constants appearing in Theorems~\ref{thm_MAST_upper} and~\ref{thm_holder_homeo}. We do not try to optimize the computations.

\paragraph{Choice of $C$ in Lemma~\ref{lem:rough_control_size_regions}.}
In \eqref{eq:chernoff inequality upper bound size regions} in the proof of Lemma~\ref{lem:rough_control_size_regions}, a quick computation shows that  $\Ec{e^{-\lambda \log(W_1)}}= \frac{1}{1-2\lambda}$ so that we are looking for $\lambda<\frac{1}{2}$ and $C>0$ such that $\frac{e^{-\lambda C}}{1-2\lambda} \leq \frac{1}{4}$.
Set $\lambda=\frac{1}{2}-\frac{1}{C}$, assuming $C>2$. We now want $\frac{C}{2}e^{1-\frac{C}{2}} \leq \frac{1}{4}$, so we can take $C=7.5$.
\paragraph{Choice of $\alpha,\beta,\gamma$ in Lemma~\ref{lem_many_good_scales}.}
To get a quantitative bound on the probability $p$ appearing in \eqref{eq:probability that next scale is good} we can use a union bound and the fact that $W_1, W_2,W_3$ all have distribution $\mathrm{Beta}(\frac{1}{2},1)$ with density $\frac{1}{2\sqrt{x}}\cdot \mathbbm{1}_{\intervalleof{0}{1}}(x)$. We get  
\begin{align*}
	\Pp{W_1, W_2,W_3 \geq \alpha} \geq 1 - 3 \Pp{W_1 <\alpha} = 1 - 3\sqrt{\alpha}.
\end{align*}
We decide to take $\alpha = 10^{-6}$. We then have $p\geq \frac{1}{9}\cdot (1 - 3\sqrt{\alpha})=\frac{0.997}{9}$ and we need $\beta< \frac{p}{2}$, so we take $\beta=\frac{1}{20}$.
Finally, to conclude the proof in a quantitative way, we use Hoeffding's inequality and we find $\gamma=(p-2\beta)^2 \approx 10^{-4}$ (its exact value won't be needed for the computations below). 

\paragraph{Choice of $\delta$ in Proposition~\ref{prop_existence_mismatches}.}
We recall from~\eqref{eq:formula for K} that the constant $K$ appearing in Lemma~\ref{lem_counting_candidates} is 
given by $K=\frac{6}{\alpha} \cdot e^{C+2} \cdot  2^{\frac{2C+4}{\beta}}$. 
Now it follows from \eqref{eq:how small delta has to be} in the proof of Proposition~\ref{prop_existence_mismatches} that it is sufficient to choose $\delta$ such that 
$6 \cdot \left( \frac{2\delta K}{\alpha\sqrt{\alpha}} \right)^{\frac{\beta}{2}}<1$.
This is equivalent to
\begin{align*}
	\delta < 6^{-\frac{2}{\beta}}\cdot\frac{\alpha^{\frac{3}{2}}}{2 K}= &\frac{6^{-1-\frac{2}{\beta}}}{2}\cdot e^{-(C+2)}\cdot \alpha^{\frac{5}{2}} \cdot 2^{-\frac{2C+4}{\beta}} 
	= \frac{6^{-41}}{2} \cdot e^{-9.5}\cdot 10^{-15} \cdot 2^{-19\cdot 20} \approx 1.89 \cdot 10^{-166},
\end{align*}
so we can take $\delta = 10^{-166}$. In particular, considering only $\beta k$ good scales is the reason why our final value is very small.

\paragraph{Choice of $\mu, \eta$ in Lemma~\ref{lem:small exponential moment for Z+c}.}
Following the proof of Lemma~\ref{lem:small exponential moment for Z+c}, we choose $\mu$ so that 
\begin{align*}
	e^{\frac12 \mu} \left( 1-\frac{\delta^2}{8} \right) \leq 1 \qquad \text{i.e.} \qquad \mu \leq -2 \log \left( 1-\frac{\delta^2}{8} \right).
\end{align*}
For example, we may choose $\mu = \frac{\delta^2}{10}=10^{-333}$. Moreover, we have defined $\eta$ as $\frac{\mu \beta}{8}=\frac{\mu}{160}>10^{-336}$.
\paragraph{Choice of $\xi$ of Proposition~\ref{prop_homeo_mass}.}
We only need $\xi < \min(\gamma,\eta)$, so we can take $\xi = 10^{-336}$.
\paragraph{Choice of $\rho$ in Corollary~\ref{prop:sum sqrt prod mass} and  $\theta, \eps_{\ref{thm_MAST_upper}}$ in Theorem~\ref{thm_MAST_upper}, and $\eps_{\ref{thm_holder_homeo}}$ in Theorem~\ref{thm_holder_homeo}.
}
In the proof of Corollary~\ref{prop:sum sqrt prod mass}, the constant $\rho$ has to be chosen so that $\rho < \frac{1}{2}\min(\eta,\xi)=\frac{\xi}{2}$. In the proof of Theorem~\ref{thm_MAST_upper}, we can take $\theta=\frac{1}{2C}=\frac{1}{15}$, and finally $\eps_{\ref{thm_MAST_upper}}< \frac{\rho\theta}{2}$, so we can take $\eps_{\ref{thm_MAST_upper}}=10^{-338}$.
Similarly $\eps_{\ref{thm_holder_homeo}}$ has to be chosen such that  $\eps_{\ref{thm_holder_homeo}} < \frac{\eta}{2C}=\frac{\eta}{15}$, so we can take $\eps_{\ref{thm_holder_homeo}}=10^{-338}$.

\subsection{Remarks and open questions}\label{subsec:discussion}

\paragraph{The expected maximum agreement subtree.}
In all the arguments of the paper, the estimates that are stated with probability $1-o_n(1)$ actually hold with probability $1-O(n^{-a})$ for some (small) $a>0$ (or with probability $1-O(e^{-ak})$, which is equivalent since we take $k$ of order $\log n$). Hence, by Theorem~\ref{thm_MAST_upper} and Lemma~\ref{lem:square root bounds on size of mast}, we can write
\[ \Pp{\MAST(T_n,T'_n) \geq n^{1/2-\eps_{\ref{thm_MAST_upper}}}} \leq n^{-a}, \quad \text{and} \quad \Pp{\MAST(T_n,T'_n) \geq Cn^{1/2}} \leq e^{-c\sqrt{n}}\]
for some constants $C,c>0$. Therefore, since $\MAST(T_n,T'_n)$ is bounded by $n$, we can write
\[ \Ec{\MAST(T_n,T'_n)} \leq n^{1/2-\eps_{\ref{thm_MAST_upper}}} + n^{-a} \times C\sqrt{n} + e^{-c\sqrt{n}} \times n = O \left( n^{1/2-\min(\eps_{\ref{thm_MAST_upper}},a)} \right),\]
so the \emph{expected} MAST is also much less than $\sqrt{n}$.

\paragraph{Other models of random trees.}
Another natural random tree model where the Maximum Agreement Subtree has been investigated~\cite{bryant_size_2003, bernstein_bounds_2015} is the Yule--Harding model $Y_n$, i.e. the model where the binary tree $Y_n$ is obtained from $Y_{n-1}$ by choosing a leaf uniformly at random and splitting it into one node and two leaves. The best known lower and upper bounds for this model are given by~\cite{bernstein_bounds_2015} and are respectively of order $n^{0.344}$ and $\sqrt{n}$.
It seems likely to us that adaptations of the ideas developed in the present paper could be used to prove Theorem~\ref{thm_MAST_upper} for the Yule--Harding model.

Beyond the binary case, another natural question would be to try to estimate the MAST between more general Galton--Watson trees. 
We believe that similar results could be obtained provided the tail of the offspring distribution is light enough (with the technical difficulty that the coupling between the discrete model and the Brownian tree would not be as simple). 
On the other hand, when the tail is very heavy, the MAST should become larger because of star-shaped subtrees.

\paragraph{Optimal regularity for homeomorphisms between continuous random structures.}
It seems natural to introduce the following quantity
\begin{align*}
	\gamma_{+} := \inf\enstq{\gamma \geq 0}{\text{there a.s. exists no homeomorphism $\Psi : \Tt\rightarrow \Tt'$ that is $\gamma$-Hölder}},
\end{align*}
which, in some sense, captures how metrically different two independent realizations of the Brownian tree are. 
Theorem~\ref{thm_holder_homeo} ensures that $\gamma_{+} \leq 1-10^{-338}$, and as it was pointed out in the introduction, Aldous's construction in \cite{aldous_largest_2022} amounts to constructing a $(5-2\sqrt{6})$-Hölder homeomorphism between $\Tt$ and $\Tt'$ so $\gamma_{+}\geq 5-2\sqrt{6} \approx 0.1010$. 
It would be an interesting direction of research to find tighter bounds on $\gamma_{+}$ or a good heuristic as to what the value of $\gamma_{+}$ may be.

The question of finding the optimal regularity for homeomorphisms between independent copies of random metric spaces can be asked for many other models. 
It is for example natural to ask whether an analog of Theorem~\ref{thm_holder_homeo} could be proved for some models with the topology of the plane such as the Brownian map~\cite{LG11, Mie11}, or more generally Liouville Quantum Gravity metrics~\cite{GM21}.

\appendix

\section{Construction of a Hölder homeomorphism}
\label{app:construction of homeomorphism}
\begin{theorem}\label{thm:existence_holder_homeomorphism}
	Let $\Tt$ and $\Tt'$ be two independent copies of the Brownian tree. 
	There almost surely exists a homeomorphism $\Psi$ from $\Tt$ to $\Tt'$ such that $\Psi$ and $\Psi^{-1}$ are both $\gamma$-H\"older, for any $\gamma < 5-2\sqrt{6}$. 
\end{theorem}
\begin{proof}
	Consider the Aldous decomposition of the tree $\mathcal{T}$, as described in Section~\ref{subsec:aldous recursive decomposition}, with the associated collection $(x[\ii])_{\ii \in \mathbb{T}_3 \backslash \{ \emptyset\}}$ of points and the collection $(b[\mathbf{i}])_{\mathbf{i}\in \mathbb T_3}$ of branching points.
	Let also $\left( \mathcal{T}', (x'[\ii])_{\ii \in \mathbb{T}_3 \backslash \{ \emptyset\}}, (b'[\mathbf{i}])_{\mathbf{i}\in \mathbb T_3} \right)$ be an independent copy of $\left( \mathcal{T}, (x[\ii])_{\ii \in \mathbb{T}_3 \backslash \{ \emptyset\}}, (b[\mathbf{i}])_{\mathbf{i}\in \mathbb T_3} \right)$, so that in particular $\Tt$ and $\Tt'$ are Brownian trees.
	We construct a map $\Psi : \enstq{b[\mathbf{i}]}{ \mathbf{i}\in \mathbb T_3} \rightarrow \enstq{b'[\mathbf{i}]}{\mathbf{i}\in \mathbb T_3}$ by setting $\Psi(b[\mathbf{i}])= b'[\mathbf{i}]$. 
	We are going to show that for any $\gamma < 5-2\sqrt{6}$, this map is almost surely $\gamma$-Hölder.
%	 By symmetry, this will also show that a.s.\ $\Psi^{-1}$ is $\gamma$-Hölder for $\gamma < 5-2\sqrt{6}$.
	Then, by density of the set of branching points that we consider, we can extend $\Psi$ in a unique way into a $\gamma$-Hölder map from $\Tt$ to $\Tt'$. 
	By symmetry, the same can be done with $\Psi^{-1}$, so the map $\Psi$ is indeed a homeomorphism, with $\Psi$ and $\Psi^{-1}$ that are $\gamma$-Hölder, almost surely.
	Since this holds almost surely for a countable sequence of values of $\gamma$ tending to $5-2\sqrt{6}$, the claim of the theorem holds. 
	
	From now on, we fix $\gamma < 5-2\sqrt{6}$ and introduce, for any $k\geq 1$,  
	\begin{align*}
		S_k := \max_{\substack{\mathbf{i}, \mathbf{j}\in \mathbb{T}_3 \\ |\mathbf{i}|, |\mathbf{j}| \leq k}} \frac{d_{\Tt'}(b'[\mathbf{i}],b'[\mathbf{j}])}{ \left(d_{\Tt}(b[\mathbf{i}],b[\mathbf{j}])\right)^\gamma} \qquad \text{and} \qquad S := \sup_{\mathbf{i}, \mathbf{j}\in \mathbb{T}_3} \frac{d_{\Tt'}(b'[\mathbf{i}],b'[\mathbf{j}])}{ \left(d_{\Tt}(b[\mathbf{i}],b[\mathbf{j}])\right)^\gamma}= \sup_{k\geq 1} S_k. 
	\end{align*}
	Proving that $\Psi$ is a.s.\ $\gamma$-Hölder continuous on $\enstq{b[\mathbf{i}]}{ \mathbf{i}\in \mathbb T_3}$ reduces to proving that $S$ is almost surely finite. 
	We prove that claim in two steps:
	\begin{itemize}
		\item First, we obtain a control on the distances between "neighbor branching points" of level $k$ that will hold true simultaneously for all pairs of such points. 
		Such points appear as pairs of the form $(x[\ii1], x[\ii2])$ in the decomposition described in Section~\ref{subsec:aldous recursive decomposition}.  
		This estimate is the content of the next lemma.
		\item Then, we use this estimate to bound $S_k-S_{k-1}$, so that $S_k$ increases to a finite limit.
	\end{itemize}
	\begin{lemma}\label{lem:distance neighbouring points}
		Almost surely, for $k$ large enough, for any $\mathbf{i}\in \mathbb{T}_3^k$, we have
		\begin{align}\label{eq:distance estimate for neighbouring points}
			d_{\Tt'} \left( x'[\mathbf{i}1],x'[\mathbf{i}2] \right) \leq k^{-2}\cdot d_{\Tt} \left( x[\mathbf{i}1],x[\mathbf{i}2] \right)^\gamma.
		\end{align}
	\end{lemma}
	\begin{proof}[Proof of Lemma~\ref{lem:distance neighbouring points}]
		%In this decomposition, for any $k\geq 1$, the whole tree $\mathcal{T}$ can be decomposed into $3^k$ bi-pointed regions 
		Recall Proposition~\ref{prop_independence_decomposition}, which states that for any $\ii\in \mathbb{T}_3^k$, conditionally on $\mathcal{R}[\mathbf{i}]$, the region $(\mathcal{R}[\mathbf{i}], x[\mathbf{i}1],x[\mathbf{i}2])$ has the distribution of a bi-pointed Brownian tree of mass $|\mathcal{R}[\mathbf{i}]|$. %Also, Proposition~\ref{} describes the distribution of  $|\mathcal{R}[\mathbf{i}]|$.
		Therefore, for any $\mathbf{i}\in \mathbb{T}_3$, conditionally on $|\mathcal{R}[\mathbf{i}]|$, the distance $d_{\Tt} \left( x[\mathbf{i}1],x[\mathbf{i}2] \right)$, is distributed as 
		\begin{align*}
			d_{\Tt} \left( x[\mathbf{i}1],x[\mathbf{i}2] \right) \overset{\mathrm{d}}{=}|\mathcal{R}[\mathbf{i}]|^{\frac{1}{2}} \cdot Y,
		\end{align*}
		where $Y$ has the Rayleigh distribution with density $4xe^{-2x^2}\ \mathrm{d}x$ on $\intervallefo{0}{\infty}$, see for example \cite[Theorem~2.11]{legall_random_2005}. 
		The same is of course true for analogous quantities in $\Tt'$.
		
		For any $\gamma \in \intervalleoo{0}{1}$ and $\ii \in \mathbb{T}_3^k$ we write
		\begin{align*}
			&\Pp{d_{\Tt'} \left( x'[\mathbf{i}1],x'[\mathbf{i}2] \right)>k^{-2}\cdot d_{\Tt} \left( x[\mathbf{i}1],x[\mathbf{i}2] \right)^\gamma}\\
			&= \Pp{|\mathcal{R}'[\mathbf{i}]|^{\frac{1}{2}} \cdot Y' > k^{-2} \cdot  \left(|\mathcal{R}[\mathbf{i}]|^{\frac{1}{2}} \cdot Y\right)^\gamma}\\
			&=\Pp{\frac{|\mathcal{R}'[\mathbf{i}]|^{\frac{1}{2}} \cdot Y'}{ \left(\mathcal{R}[\mathbf{i}]|^{\frac{1}{2}} \cdot Y\right)^\gamma}> k^{-2}}\\
			&\leq k^{2\lambda} \cdot \Ec{|\mathcal{R}'[\mathbf{i}]|^{\frac{\lambda}{2}}} \cdot \Ec{|\mathcal{R}[\mathbf{i}]|^{-\frac{\lambda \gamma }{2}}} \cdot \Ec{(Y')^\lambda} \cdot \Ec{Y^{-\gamma \lambda}},
		\end{align*}
		where in the last line we raised everything to some power $\lambda<\frac{2}{\gamma}$ and used the Markov inequality. The condition $\lambda<\frac{2}{\gamma}$ ensures that the quantity $C(\gamma,\lambda):=\Ec{(Y')^\lambda} \cdot \Ec{Y^{-\gamma \lambda}}$ is finite.
		On the other hand, it follows from Proposition~\ref{prop_independence_decomposition} that $|\mathcal{R}[\ii]|$ is the product of $k$ i.i.d.\ $\mathrm{Beta}(1/2,1)$-random variables. 
		Since the $p$-th moment of such a Beta variable is given by $\frac{1}{2p+1}$ for $p>-\frac{1}{2}$, we get
		\begin{align*}
			\Ec{|\mathcal{R}'[\mathbf{i}]|^{\frac{\lambda}{2}}}\Ec{|\mathcal{R}[\mathbf{i}]|^{-\frac{\lambda \gamma }{2}}} =\left( \frac{1}{(1+\lambda)(1-\gamma \lambda)}\right)^k.
		\end{align*}
		The expression $(1+\lambda)(1-\gamma \lambda)$ appearing in the denominator is maximized in $\lambda$ at $\lambda = \frac{1-\gamma}{2\gamma}<\frac{2}{\gamma}$ and the value of the maximum is $1+\frac{(1-\gamma)^2}{4\gamma}$. 
		This value decreases in $\gamma$ and attains $3$ at $5- 2\sqrt{6}$. 
		Hence, for $\gamma < 5-2\sqrt{6}$ we can write a union bound over all $\mathbf{i} \in \mathbb{T}_3^k$ to get  
		\begin{align*}
			\Pp{ \exists \mathbf{i} \in \mathbb{T}_3^k : d_{\Tt'} \left( x'[\mathbf{i}1],x'[\mathbf{i}2] \right)>k^{-2}\cdot d_{\Tt} \left( x[\mathbf{i}1],x[\mathbf{i}2]\right)^\gamma} \leq k^{\frac{1-\gamma}{\gamma}} \cdot\left(\frac{3}{1+\frac{(1-\gamma)^2}{4\gamma}} \right)^k \cdot C(\gamma, \textstyle\frac{1-\gamma}{2\gamma}),
		\end{align*}
		which is summable in $k\geq 1$. This ensures using the Borel-Cantelli lemma that our lemma holds true.
	\end{proof}

	From now on, we let $K_0$ be the first time for which \eqref{eq:distance estimate for neighbouring points} holds for all $k \geq K_0$, and fix $k \geq K_0$.
	%The goal of the end of the proof will be to prove that $\sup_{k\geq 1} \alpha_k<\infty$. 
	%Remark that by construction we have the following equality as sets
	%\begin{align*}
	%	 \enstq{x[\mathbf{i}]}{ \mathbf{i}\in \mathbb T_3^k} = \enstq{x[\mathbf{i}a]}{ \mathbf{i}\in \mathbb T_3^k, \ a \in \{1,2\}}.
	%\end{align*}
	Consider two indices in $\mathbb{T}_3$, say $\mathbf{i}$ and $\mathbf{j}$, such that $|\ii|,|\jj| \leq k$,  and consider the path $p_{\ii \to \jj}$ in the tree $\Tt$ going from $b[\ii]$ to $b[\jj]$. 
	We denote by $\mathbf{l}$  the index such that $b[\mathbf{l}]$ is the first point of $\enstq{b[\mathbf{i}]}{ \mathbf{i}\in \mathbb T_3, \ |\ii| \leq k-1}$ visited by $p_{\ii \to \jj}$ (note that this index can possibly be $\ii$ itself, if $|\ii| \leq k-1$).
	Similarly we let $\mathbf{m}$ be so that $b[\mathbf{m}]$ is the last such point. 
	We can then write
	\begin{align*}
		d_{\Tt}(b[\mathbf{i}],b[\mathbf{j}]) = d_{\Tt}(b[\mathbf{i}],b[\mathbf{l}])+d_{\Tt}(b[\mathbf{l}],b[\mathbf{m}])+d_{\Tt}(b[\mathbf{m}],b[\mathbf{j}]).
	\end{align*}
	Now, either $b[\mathbf{i}]=b[\mathbf{l}]$, in which case their distance is $0$, or $b[\mathbf{i}]\neq b[\mathbf{l}]$ and in that case this pair of branching points can be written as $\{b[\mathbf{i}], b[\mathbf{l}]\} =\{x[\mathbf{u}1],x[\mathbf{u}2]\}$ for some $\mathbf{u}\in \mathbb{T}_3$ with $|\mathbf{u}|=k+1$. Indeed, two consecutive branching points of depth at most $k$ along $p_{\ii \to \jj}$ are always of this form, and it always holds that one has depth exactly $k$ and the other at most $k-1$, so $b[\mathbf{l}]$ is either $b[\ii]$ or the second branching point of $p_{\ii \to \jj}$. The same considerations hold in $\Tt'$ so in any case, by Lemma~\ref{lem:distance neighbouring points}, we have
	\begin{align*}
		d_{\Tt'} \left( b'[\mathbf{i}],b'[\mathbf{l}] \right) \leq (k+1)^{-2} \cdot d_{\Tt} \left( b[\mathbf{i}],b[\mathbf{l}] \right)^\gamma,
	\end{align*}
	and the analogous inequality is true for $d_{\Tt'}(b'[\mathbf{m}],b'[\mathbf{j}])$.
	
	Therefore, for $k\geq K_0$, we have 
	\begin{align*}
		&d_{\Tt'} \left( b'[\mathbf{i}],b'[\mathbf{j}] \right) \\
		&= d_{\Tt'} \left( b'[\mathbf{i}],b'[\mathbf{l}] \right)+d_{\Tt'} \left( b'[\mathbf{l}],b'[\mathbf{m}] \right)+d_{\Tt'} \left( b'[\mathbf{m}],b'[\mathbf{j}] \right) \\
		&\leq (k+1)^{-2}\cdot d_{\Tt} \left( b[\mathbf{i}],b[\mathbf{l}]\right)^\gamma + S_{k-1} \cdot d_{\Tt} \left( b[\mathbf{l}],b[\mathbf{m}]\right)^\gamma + (k+1)^{-2}\cdot d_{\Tt} \left( b[\mathbf{m}],b[\mathbf{j}] \right)^\gamma\\
		&\leq \left( (k+1)^{-\frac{2}{1-\gamma}} + S_{k-1}^{\frac{1}{1-\gamma}} +(k+1)^{-\frac{2}{1-\gamma}}\right)^{1-\gamma} \cdot \left( d_{\Tt}(b[\mathbf{i}],b[\mathbf{l}])+d_{\Tt}(b[\mathbf{l}],b[\mathbf{m}])+d_{\Tt}(b[\mathbf{m}],b[\mathbf{j}])\right)^\gamma\\
		&\leq \left( S_{k-1}^\frac{1}{1-\gamma} +2(k+1)^{-\frac{2}{1-\gamma}} \right)^{1-\gamma} \cdot d_{\Tt} \left( b[\mathbf{i}],b[\mathbf{j}]\right)^\gamma,
	\end{align*}
	where we used the Hölder inequality with $p=\frac{1}{1-\gamma}$ and $q=\frac{1}{\gamma}$.
	This ensures that for $k\geq K_0$, we have
	\begin{align*}
		S_{k}\leq \left( S_{k-1}^\frac{1}{1-\gamma} +2(k+1)^{-\frac{2}{1-\gamma}}\right)^{1-\gamma} \leq S_{k-1}+\frac{2}{(k+1)^2},
	\end{align*}
	and so $\sup_{k\geq 1} S_k<\infty$. This concludes the proof.
\end{proof}

\bibliographystyle{siam}
\bibliography{biblio_mast.bib}
\end{document}